\documentclass{article}
\usepackage{amsmath, amsthm,amsfonts, amssymb, graphicx, cite, epsfig,epstopdf,color,soul,comment,graphicx}

\usepackage{enumerate}
\usepackage[shortlabels]{enumitem}
\usepackage[ruled]{algorithm2e}
\usepackage[titletoc,page]{appendix}
\usepackage{enumerate}
\usepackage[shortlabels]{enumitem}
\usepackage{empheq}
\usepackage{geometry}%
\usepackage[bookmarks=true,pageanchor,colorlinks,linkcolor=red,anchorcolor=blue, citecolor=blue,urlcolor=blue,hyperfootnotes=false]{hyperref}

\newcommand{\Eset}{\mathbb{E}}

\newcommand{\Nset}{\mathbb{N}}

\newcommand{\Rset}{\mathbb{R}}

\newcommand{\Ecal}{{\cal E}}
\newcommand{\Fcal}{{\cal F}}
\newcommand{\Gcal}{{\cal G}}

\newcommand{\Kcal}{{\cal K}}

\newcommand{\Ncal}{{\cal N}}

\newcommand{\Rcal}{{\cal R}}
\newcommand{\Scal}{{\cal S}}

\newcommand{\Ucal}{{\cal U}}
\newcommand{\Vcal}{{\cal V}}

\newcommand{\Abf}{{\bf A}}
\newcommand{\Bbf}{{\bf B}}

\newcommand{\Dbf}{{\bf D}}

\newcommand{\Ibf}{{\bf I}}

\newcommand{\Pbf}{{\bf P}}
\newcommand{\Qbf}{{\bf Q}}

\newcommand{\Ubf}{{\bf U}}

\newcommand{\Wbf}{{\bf W}}

\newcommand{\Ybf}{{\bf Y}}

\newcommand{\btheta}{{\bar{\theta}}}

\newcommand{\1}{{\mathbf{1}}}

\newcommand{\bbar}{{\bar{b}}}

\newtheorem{lemma}{Lemma}
\newtheorem{theorem}{Theorem}

\newtheorem{definition}{Definition}

\newtheorem{assump}{Assumption}
\theoremstyle{remark}
\newtheorem{remark}{Remark}

\usepackage{color}

\title{Finite-Time Performance of Distributed Temporal Difference Learning with Linear Function Approximation}
\author{{Thinh T. Doan, Siva Theja Maguluri, Justin Romberg\thanks{Thinh T. Doan and Siva Theja Maguluri are with the School of Industrial and Systems Engineering, and Justin Romberg is with the School of Electrical and Computer Engineering, Georgia Institute of Technology, GA, 30332, USA. Email:{\tt\small \{thinhdoan,\,siva.theja\}@gatech.edu, jrom@ece.gatech.edu.}}}} 
\date{}
\begin{document}
\maketitle

\begin{abstract}
We study the policy evaluation problem in multi-agent reinforcement learning, modeled by a Markov decision process.
In this problem, the agents operate in a common environment under a fixed control policy, working together to discover the value (global discounted accumulative reward) associated with each environmental state. 

Over a series of time steps, the agents act, get rewarded, update their local estimate of the value function, then communicate with their neighbors.  The local update at each agent can be interpreted as a distributed variant of the popular temporal difference learning methods {\sf TD}$  (\lambda)$. 

Our main contribution is to provide a finite-analysis on the performance of this distributed {\sf TD}$(\lambda)$ algorithm for both constant and time-varying step sizes. The key idea in our analysis is to use the geometric mixing time $\tau$ of the underlying Markov chain, that is, although the ``noise" in our algorithm is Markovian, its dependence is very weak at samples spaced out at every $\tau$.  We provide an explicit upper bound on the convergence rate of the proposed method as a function of the network topology, the discount factor, the constant $\lambda$, and the mixing time $\tau$. 

Our results also provide a mathematical explanation for observations that have appeared previously in the literature about the choice of $\lambda$.  Our upper bound illustrates the trade-off between approximation accuracy and convergence speed implicit in the choice of $\lambda$.  When $\lambda=1$, the solution will correspond to the best possible approximation of the value function, while choosing $\lambda = 0$ leads to faster convergence when the noise in the algorithm has large variance.
\end{abstract}


\section{Introduction}\label{sec:intro}

Reinforcement learning {(\sf RL)} is a general paradigm for learning optimal policies in stochastic control problems based on simulation \cite{Sutton1998_book,Bertsekas1999_book,Szepesvari2010_book}.  Relatively simple {\sf RL} algorithms have been remarkably effective in solving challenging practical problems; notable examples include autonomous driving \cite{Chen2015_DeepDrive}, robotics \cite{Gu2017_DeepRL}, helicopter flight \cite{Abbeel2007_NIPS}, board games \cite{Silver2016_MasteringTG}, and power networks \cite{Kar2013_QDLearning}.  In these applications, an agent tries to find an optimal policy (a mapping from state to action) through interactions with the environment, modeled as a Markov Decision Process {(\sf MDP)}, with the goal of optimizing its long-term future reward (or cost).  One of the central problems in {\sf RL}, and the one considered in this paper, is the {\em policy evaluation problem}, where the expected long-term reward for a given stationary policy is estimated.  While interesting in its own right, policy evaluation also arises as a subproblem in {\sf RL} policy search methods, including policy iteration and actor-critic methods \cite{Sutton1998_book,Bertsekas1999_book}.   
Temporal-difference learning ({\sf TD}$(\lambda)$), originally proposed by Sutton \cite{Sutton1988_TD}, is one of the most efficient and practical methods for policy evaluation.  The iterations for {\sf TD}$(\lambda)$, discussed in detail in Section~\ref{subsec:DistributedTD} below, are relatively simple, and can be implemented in an online fashion.  The long-term future cost is estimated as a function of the current state, with the scalar parameter $\lambda\in[0,1]$ controlling the trade-off between the accuracy of the approximation and the susceptibility to simulation noise.  The convergence of {\sf TD}$(\lambda)$ is a classical result, and has been analyzed under various sets of assumptions in \cite{Dayan1992_TD,Gurvits1994_TD,Pineda1997_TD,Tsitsiklis1997_TD, Tsitsiklis1999_TDAverage}.  In practice, the domain of the policy value function (the state space) can be very large.  To avoid (or at least mitigate) the ``curse of dimensionality'', the value function is chosen from a parametric set \cite{Tesauro1995_TD0,Mnih2015_Nature,Silver2016_MasteringTG}.  The particular case we analyze here is linear function approximation, where the value function at each state is a linear combination of (fixed) feature vectors, which allows us to explicitly characterize its rate of convergence.

In this paper, we study the policy evaluation problem where multiple agents explore a common environment.  We are motivated by broad applications of the multi-agent paradigm within engineering, for example, mobile sensor networks \cite{CortesMKB2004,Ogren2004_TAC}, cell networks \cite{Bennis2013_CellNetworks}, and power networks \cite{Kar2013_QDLearning}.  In this multi-agent reinforcement learning ({\sf MARL}) setting, each agent takes its own action based on the current (global) state, and the agents receive different local rewards that are a function of their current state, their new state, and their action.  We assume that each agent observes only its own reward.  Their goal is to cooperatively evaluate the global accumulative reward while only communicating with a small subset of the other agents.  We introduce a distributed variant of {\sf TD}$(\lambda)$ algorithm with linear function approximation, and consider data samples generated by the Markov process representing the environment. Thus, the noise in our algorithm is dependent since it is Markovian.  Our results provide a finite-time analysis of distributed {\sf MARL} {\sf TD}$(\lambda)$, giving an upper bound on the distance of the current estimate from the eventual point of convergence as a function of number of iterations (and other problem parameters).  The results in this paper generalize our preceding work \cite{DoanMR2019_DTD(0)}, where we studied distributed {\sf TD}$(0)$ under i.i.d sampling assumption (i.e, the noise in the algorithm is i.i.d). 

\subsection{Existing literature}
In general, {\sf TD}$(\lambda)$ with linear function approximation can be viewed as a variant of the celebrated stochastic approximation ({\sf SA}) method, whose asymptotic convergence is typically analyzed using the so-called ordinary differential equation ({\sf ODE}) method \cite{borkar2008}. The {\sf ODE} method shows that under the right conditions, the effects of noise eventually average out and the {\sf SA} iterates asymptotically follow the trajectory of a stable {\sf ODE}. In particular, Tsitsiklis and Van Roy considered a policy evaluation problem on a discounted {\sf MDP} for both finite and infinite state spaces with linear function approximation \cite{Tsitsiklis1997_TD}.  By viewing {\sf TD} as a stochastic approximation for solving a suitable Bellman equation, they characterized its almost-sure convergence using the {\sf ODE} approach. Following this work, Borkar and Meyn provided a general and unified framework for the convergence of {\sf SA} with broad applications in {\sf RL} \cite{Borkar2000_ODE}.  General results in this area can be found in the monograph by Borkar \cite{borkar2008}.

While the {\sf ODE} method can be used to study the asymptotic convergence of {\sf TD}$(\lambda)$, it is not obvious how to derive a rate of convergence using this approach. In general, {\sf TD}$(\lambda)$ does not correspond to stochastic gradient descent ({\sf SGD}) on any static optimization problem,  making it challenging to characterize the consistency and quantify the progress of this method. 
Indeed, the convergence rates of TD$(\lambda)$ remained largely open until recently \cite{Thoppe2018_ConcentrationBound,Dalal2018_FiniteTD0,Bhandari2018_FiniteTD,LakshminarayananS2018,LeeH2019,SrikantY2019_FiniteTD,HuS2019,ChenZDMC2019}. In particular, a concentration bound was given in \cite{Thoppe2018_ConcentrationBound} for the {\sf SA} algorithm under a strict stability assumption of the iterates. A finite-time analysis of the {\sf TD} method with linear function approximation was simultaneously studied in \cite{Dalal2018_FiniteTD0, Bhandari2018_FiniteTD,LakshminarayananS2018,LeeH2019} for a single agent. These works carefully characterize the progress of the {\sf TD} update and derive its convergence rate using standard techniques from the study of {\sf SGD} and the results in \cite{Tsitsiklis1997_TD}. Recently, the  work in \cite{SrikantY2019_FiniteTD} provides a finite-time error bound of linear {\sf SA} under Markovian noise and constant step sizes under very general conditions (e.g., without requiring an additional projection step to keep the iterates bounded). Motivated by this work, the authors in \cite{HuS2019} study the finite-time performance of the linear {\sf SA} using control theoretic approach, while the work in \cite{ChenZDMC2019} provides the finite-time convergence of nonlinear {\sf SA}.  

Within the context of {\sf MARL}, an asymptotic convergence of a distributed gossiping version of {\sf TD}$(0)$ with linear function approximation was probably first studied in \cite{Mathkar2017_GossipRL}, where again convergence is established using the {\sf ODE} approach.  Similar results were also studied implicitly in \cite{zhang2018_ICML}, which analyzes the convergence of distributed actor-critic methods.  While finite-time (convergence rate) analysis exists for single agent problems \cite{Dalal2018_FiniteTD0,Bhandari2018_FiniteTD,LakshminarayananS2018,SrikantY2019_FiniteTD,HuS2019}, a rate of convergence of distributed {\sf TD}$(\lambda)$  does not exists in the current literature, and is the focus of this paper. While the convergence rates of distributed {\sf TD}$(0)$ were studied in our earlier work \cite{DoanMR2019_DTD(0)}, the results were derived under (perhaps unrealistically strong) assumptions on the noise being independent, and the algorithm included a projection step at each iterate that required a priori knowledge about the solution. Below, we study the rate of convergence of {\sf TD}$(\lambda)$ under Markovian noise and without requiring any projection step. Our approach is motivated by the recent work in \cite{SrikantY2019_FiniteTD} about the rates of {\sf TD}$(\lambda)$ for a single agent problem.           

Finally, we mention some related {\sf RL} methods for solving policy evaluation problems in both single agent {\sf RL} and {\sf MARL}, such as, the gradient temporal difference methods studied in \cite{Sutton2008_GTD,Sutton2009_FGTD,Liu2015_FiniteGTD,Macua_2015_TAC,Stankovic2016_ACC,Wai2018_NIPS}, least squares temporal difference ({\sf LSTD}) \cite{Bradtke1996_LSTD,Tu2018_ICML}, and least squares policy evaluation ({\sf LSPE}) \cite{Nedic2003_LSPE,Yu2010_ICML,Yu2009_TAC}. Although they share some similarity with {\sf TD} learning, these methods belong to a different class of algorithms whose update iterations are more complex. Our analysis focuses on a decentralized extension of {\sf TD} online learning that is practical and simple to implement (although, as we will see below, not as straightforward to analyze). 

\subsection{Main contributions}
In this paper, we study a distributed variant of the temporal difference learning method for solving the policy evaluation problem in multi-agent reinforcement learning. Our distributed algorithm is composed of a consensus step amongst locally communicating agents, followed by a local {\sf TD}$(\lambda)$ update. Our main contribution is to provide a finite-time analysis on the performance of this distributed {\sf TD} algorithm for both constant and time-varying step sizes. The key idea in our analysis is to properly utilize the geometric mixing time $\tau$ of the underlying Markov chain, which allows a systematic treatment of the Markovian ``noise'', as its dependence becomes quantifiably weak for observations $\tau$ steps apart.
%

We provide an explicit upper bound on the rate of convergence of the proposed method as a function of the network topology, discount factor, the constant $\lambda$, and the mixing time $\tau$ of the underlying Markov chain. Our results theoretically address some numerical observations of {\sf TD}$(\lambda)$ , that is, $\lambda=1$ gives the best approximation of the function values while $\lambda = 0$ leads to better performance when there is large variance in the algorithm. 



\section{Multi-agent reinforcement learning}
We consider a multi-agent reinforcement learning system of $N$ agents operating in a common environment, modeled by a Markov decision process. We assume that the agents can communicate with each other through a connected and undirected graph $\Gcal = (\Vcal,\Ecal)$, where $\Vcal = \{1,\ldots,N\}$ and $\Ecal = \Vcal\times\Vcal$ are the vertex and edge sets, respectively. At each time step, the agents observe the state of the environment, take an action based on this observation, and receive a corresponding award.
The goal of the agents is to cooperatively estimate the global discounted accumulative reward, which is composed of local rewards received by the agents. This is the policy evaluation problem for multi-agent systems, which can be mathematically characterized as follows. 

Let $\Scal$ be the global finite state space of the environment, where $S$ denotes its cardinality. In addition, we denote by $\Ucal^{v}$ the set of control actions at each agent $v$. At each time step $k\geq 0$, each agent $v\in\Vcal$ observes the state $s_{k}\in\Scal$ of the environment and based on that observation take an action $u_{k}^{v} = \mu^{v}(s_k) \in \Ucal^{v}$. Here $\mu^{v}:\Scal\rightarrow\Ucal^{v}$, the policy of agent $v$, is a function mapping the state of the environment to a control action in $\Ucal^{v}$. We denote by $\Ucal = \Ucal^{1}\times\ldots\times \Ucal^{N}$ and let $u = u^{1}\times\ldots\times u^{N}\in \Ucal$ be the joint action of the agents. As a consequence of their joint actions, the environment moves to a new state $s_{k+1}\in\Scal$. In addition, each agent $v$ receives an instantaneous local reward $\Rcal^v_{k} = \Rcal^{v}(s_{k},u_{k},s_{k+1})$. The goal of the agents is to cooperatively find the average of total discounted accumulative reward $J$ over the network defined as 
\begin{align}
J(i) \triangleq \Eset \left[\sum_{k=0}^{\infty}\frac{\gamma^{k}}{N}\sum_{v\in\Vcal}\Rcal^{v}(s_{k},u_{k},s_{k+1})\Bigm| s_{0} = i\right],  \label{distributed:ValueFunc}
\end{align}
where $\gamma$ is the discount factor. 

The system evolves under a fixed, stationary policy at each agent, and so the dynamics of the environment can be modeled by a Markov chain. We use $\Pbf$ for the transition probabilities for this underlying Markov chain: $\Pbf(s_{k} = i,s_{k+1} = j) = p_{ij}$ for $i,j\in\Scal$ is the probability of the environment transitioning to state $j$ from state $i$. (These transition probabilities also of course depend on the policy, but since the policy remains fixed, we do not include it in the notation.)  It is well-known that $J$ satisfies the Bellman equation \cite{Bertsekas1999_book,Sutton1998_book},
\begin{align}
J(i) = \sum_{j=1}^n p_{ij} \left\{\frac{1}{N}\sum_{v\in\Vcal}\Rcal^{v}(i,j)+\gamma J(j) \right\},\qquad i\in\Scal.\label{distributed:Bellman}
\end{align} 
We are interested in the case when the number of states is very large, and so computing $J$ exactly may be intractable.  To mitigate this, we use low-dimensional approximation $\tilde{J}$ of $J$, restricting $\tilde{J}$ to be in a linear subspace.  
While more advanced nonlinear approximations using, for example, neural nets as in the recent works \cite{Mnih2015_Nature, Silver2016_MasteringTG} may lead to more powerful approximations, the simplicity of the linear model allows us to analyze it in detail.  The linear function approximation $\tilde{J}$ is parameterized by a weight vector $\theta\in\Rset^{L}$, with\vspace{-0.2cm}
\begin{align}
\tilde{J}(i,\theta) = \sum_{\ell=1}^{L}\theta_\ell\phi_{\ell}(i),\label{centralized:LinearAprox}
\end{align}  
for a given set of $L$ basis vectors $\phi_{\ell}:\Scal\rightarrow\Rset$, $\ell\in\{1,\ldots,L\}$. 
We are interested in the case  $L\ll S$. Let $\phi(i)$ be the feature vector defined as\footnote{We refer to $\phi_{\ell}\in\Rset^{S}$ as the basis vectors and $\phi(i)\in\Rset^L$ as the feature vectors since it is defined on the environmental states.}
\begin{align*}
\phi(i) = (\phi_1(i),\ldots,\phi_{L}(i))^{T}\in\Rset^{L}.
\end{align*}
And let $\Phi\in\Rset^{S\times L}$ be a matrix, whose $i$-th row is the row vector $\phi(i)^{T}$ and whose $\ell$-th column is the vector $\phi_\ell = (\phi_\ell(1),\ldots,\phi_\ell(S))^{T}\in\Rset^{S}$, that is 
\begin{align*}
\Phi = \left[\begin{array}{ccc}
\vert     &  & \vert\\
\phi_1     & \cdots & \phi_L\\
\vert & & \vert
\end{array}\right] = \left[\begin{array}{ccc}
\mbox{---}   & \phi(1)^{T} & \mbox{---}\\
\cdots     & \cdots & \cdots\\
\mbox{---}   & \phi(S)^{T} & \mbox{---}
\end{array}\right]\in \Rset^{S\times L}.
\end{align*}
Thus, $\tilde{J}(\theta) =  \Phi \theta$.  The gradient of $\tilde{J}(i,\theta)$ and the Jacobian of $\tilde{J}(\theta)$ with respect to $\theta$ are
\begin{align*}
\nabla\tilde{J}(i,\theta) = \phi(i),
\quad\text{and}\quad
\nabla\tilde{J} = \Phi^{T}.
\end{align*}
The goal now is to find a $\tilde{J}$ that is the best approximation of $J$  based on the generated data by applying the stationary policy $\mu = (\mu^{1},\ldots,\mu^{N})$ on the {\sf MDP}. That is, we seek a weight $\theta^*$ such that the distance between $\tilde{J}$ and $J$ is minimized. In our setting, since each agent knows only its own reward function, the agents have to cooperate to find $\theta^*$. 
In the next section, we solve this problem using a distributed variant of the {\sf TD}$(\lambda)$ algorithm \cite{Sutton1988_TD}, where the agents only share their estimates of the weight $\theta^*$ to their neighbors but not their local rewards.

\subsection{Distributed TD$(\lambda)$ methods}\label{subsec:DistributedTD}

We introduce a consensus-based variant of the centralized {\sf TD}$(\lambda)$ method to find the weight $\theta^*$; the method is stated formally in Algorithm \ref{alg:DTD}. Our algorithm is a distributed variant of stochastic approximation for solving a suitably reformulated Bellman equation for \eqref{distributed:ValueFunc} \cite{Sutton1998_book,Bertsekas1999_book}. Algorithm \ref{alg:DTD} can be explained as follows.

Each agent $v$ maintains its own estimate $\theta_v\in\Rset^{L}$ of the weight $\theta^*$. At every iteration $k\geq0$, each agent $v$ first performs a weighted average to combine its own estimate and the ones received from its neighbors $u\in\Ncal_{u} := \{ u \in \Vcal\; |\; (v,u) \in\Ecal\}$,
\begin{align*}
y_{k}^{v} &= \sum_{u\in\Ncal_{v}}W_{vu}\theta_{k}^u.
\end{align*}
Agent $v$ then computes the so-called local temporal difference $d_k^{v}$ based on the 
observed data $(s_{k},s_{k+1},r_{k}^v)$ returned by the environment, 
\begin{align*}
d_{k}^{v} &= r_{k}^{v} + \left(\gamma\phi(s_{k+1}) - \phi(s_{k})\right)^{T}\theta_{k}^{v}.
\end{align*}
Here, $d_{k}^{v}$ represents the difference between the local outcome at agent $v$, $r_{k}^{v} + \gamma\tilde{J}(s_{k+1},\theta_{k}^{v})$ and the current estimate $\tilde{J}(s_{k},\theta_{k}^v)$. Using $y_{k}^{v}$ and $d_{k}^{v}$ agent $v$ then updates its estimate $\theta^{v}$ as
\begin{align}
\theta_{k+1}^{v} &= y_{k}^{v} + \alpha_{k}d_{k}^{v}z_{k}^{v},    \label{centralized:TD_ld}
\end{align}
where $\alpha_{k}$ is a non-negative step size and $z^{v}_{k}$ is the so-called eligibility (or trace) vector.  The vector $z^{v}_{k}$ is a weighted combination of all observed feature vectors up to time $k$,
\begin{align}
    \label{eq:zeligibility}
    z_{k}^{v} = \sum_{u=0}^{k}(\gamma\lambda)^{k-u}\phi(s_u).    
\end{align}
In addition, the local temporal differences provide the agents an indicator for increasing or decreasing their current state values (by adjusting $\theta^{v}$) after each transition. Finally, each node $v$ returns a time-weighted average $\hat{\theta}_{k}^{v}$  
\begin{align*}
\hat{\theta}_{k}^{v} = \frac{\sum_{t=0}^{k}\alpha_{t}\theta_{k}^{v}}{\sum_{t=0}^{k}\alpha_{t}}\cdot
\end{align*}
\begin{algorithm}
\caption{Distributed {\sf TD}$(\lambda)$ Algorithm}
\begin{enumerate}[leftmargin = 4mm]
\item \textbf{Initialize}: Each agent $v$ arbitrarily initializes $\theta_{0}^{v}\in\Rset^{L}$, $z_{-1}^{v} = 0$, and the sequence of stepsizes $\{\alpha_{k}\}_{k\in\Nset}$. Set $\hat{\theta}_{0}^{v} = \theta_{0}^v$ and $S_{0}^{v} = 0$.
\item \textbf{Iteration}: For $k=0,1,\ldots,$ agent $v\in\Vcal$ implements
\begin{itemize}[leftmargin = 4mm]
    \item[a.] Exchange $\theta_{k}^{v}$ with agent $u\in\Ncal_{v}$
    \item[b.] Observe a tuple $(s_{k},s_{k+1},r_{k}^v)$
    \item[c.] Execute local updates 
        \begin{align}
        \begin{aligned}
        y_{k}^{v} &= \sum_{u\in\Ncal_{v}}W_{vu}\theta_{k}^u\\ 
        d_{k}^{v} &= r_{k}^{v} + \left(\gamma\phi(s_{k+1}) - \phi(s_{k})\right)^{T}\theta_{k}^{v}\\
        \theta_{k+1}^{v} &= y_{k}^{v} + \alpha_{k}d_{k}^{v}z_{k}^{v}\\
        z_{k+1}^{v} &= \gamma\lambda z_{k}^{v} + \phi(s_{k+1})
        \end{aligned}
        \label{alg:theta_v}
        \end{align}
    \item[d.] Update the output
    \begin{align*}
     S_{k+1}^{v} &= S_{k}^{v} + \alpha_{k+1}\notag\\
    \hat{\theta}_{k+1}^{v} &= \frac{S_{k}^{v}\hat{\theta}_{k}^{v} + \alpha_{k+1}\theta_{k+1}^{v}}{S_{k+1}^{v}}
    \end{align*}
\end{itemize}
\end{enumerate}\label{alg:DTD}
\end{algorithm}


The updates in Eq.\ \eqref{alg:theta_v} have a simple interpretation: agent $v$ first computes $y_v$ by forming a weighted average of its own value $\theta_v$ and the values $\theta_u$ received from its neighbor $u\in\Ncal_v$, with the goal of seeking consensus on their estimates. Agent $v$ then moves along its own temporal direction $d_{k}^{v}z_{k}^{v}$ to update its estimate, pushing the consensus point toward $\theta^*$. In Eq.\ \eqref{alg:theta_v} each agent $v$ only shares $\theta^v$ with its neighbors but not its immediate reward $r^v$. In a sense, the agents implement in parallel $N$ local {\sf TD}$(\lambda)$ methods and then combine their estimates through consensus steps to find the global approximate reward $\tilde{J}$.


\subsection{Characterization of $\theta^*$}
The convergence of Algorithm \ref{alg:DTD} to a point $\theta^*$ for the case of single agent was established for all $0\leq\lambda\leq 1$ in \cite{Tsitsiklis1997_TD}.  We review below some of the properties of $\theta^*$.
Let $\pi = (\pi(1),\ldots,\pi(S))$ be the stationary distribution associated with the transition matrix 
$\Pbf$, which we assume exists and is unique, and $\Dbf\in\Rset^{S\times S}$ be the diagonal matrix whose diagonal entries are $\pi(i)$ for $i\in\Scal$.  Let $\Pi J$ be the projection of a vector $J$ to the linear subspace spanned by the basis vectors $\phi_\ell$,
\begin{align*}
\Pi J = \underset{y\in\text{span}\{\phi_\ell\}}{\arg\min}\|y - J\|_{\Dbf},    
\end{align*}
where $\|J\|_{\Dbf}^2 = J^{T}\Dbf J$ is the weighted norm of $J$ associated with $\Dbf$. Moreover, let $\Ubf\in\Rset^{S\times S}$ and $b^{v}\in\Rset^{L}$ for all $v\in\Vcal$ be defined as 
\begin{align*}
\Ubf &= (1-\lambda)\sum_{k=0}^{\infty}\lambda^{k}(\gamma\Pbf)^{k+1}\\
b^{v} &= \Phi^{T}\Dbf\sum_{k=0}^{\infty}(\gamma\lambda\Pbf)^{k}r^{v},\qquad \forall v\in\Vcal, 
\end{align*}
where $r^{v}\in\Rset^{S}$ is a vector whose $i$-th component is $r^v(i) = \sum_{j=1}^np_{ij}\Rcal^v(i,j)$ for all $v\in\Vcal$. Then, by \cite{Tsitsiklis1997_TD} we have that the  weight $\theta^*$ satisfies 
\begin{align}
\Abf\theta^* = b \triangleq \frac{1}{N} \sum_{v\in\Vcal} b^{v},\label{opt_cond}
\end{align}
where $\Abf$ is a negative definite matrix , i.e., $x^{T}\Abf x < 0$ $\forall x\in\Rset^{L}$, satisfying
\begin{align}
    \Abf = \Phi^{T}\Dbf(\Ubf-\Ibf)\Phi\in\Rset^{L\times L}.  \label{notation:A}
\end{align}
In addition, $\theta^*$ also satisfies  
\begin{align}
\|\Pi J - J\|_{D} \leq \|\Phi \theta^* - J\|_{D} \leq \frac{1-\gamma\lambda}{1-\gamma}\|\Pi\, J - J\|_D.\label{theta*:opt_bound}   
\end{align}
As can be seen from \eqref{theta*:opt_bound}, when $\lambda=1$ we obtain the best approximation value, i.e., $\Phi\theta^* = \Pi J$. However, in the next section we show that although $\lambda < 1$ gives a suboptimal solution it converges faster when the noise due to sampling in our algorithm has large variance.  

\section{Finite-time performance of distributed {\sf TD}$(\lambda)$}\label{sec:results}
In this section, our goal is to provide a finite-time analysis for the convergence of the distributed {\sf TD}$(\lambda)$ presented in Algorithm \ref{alg:DTD}. Motivated by the recent work \cite{SrikantY2019_FiniteTD} for the single agent problem, we provide an explicit formula for the upper bound on the rates of distributed {\sf TD}$(\lambda)$ for both constant and time-varying step sizes. The key idea in our analysis is to use the geometric mixing time $\tau$ of the underlying Markov chain; although the ``noise" our algorithm encounters is Markovian, its dependence is very weak for samples at least $\tau$ steps apart. 
A careful characterization of the coupling between these iterates allows us to derive an explicit formula for the convergence rate of distributed {\sf TD}$(\lambda)$.

\subsection{Notation and Assumptions}
Let $X_{k} = (s_{k},s_{k+1},z_{k})$ be the Markov chain defined by Algorithm~\ref{alg:DTD}, where the discrete-valued $s_k\in\Scal$ have transition probabilities $\Pbf$ and the continuous-valued $z_{k}^{v}$ obeys \eqref{eq:zeligibility}.  While the $z_{k}^{v}$ are computed locally at each agent, they only depend on the global state sequence, and so will be identical for all $v$; below we will use $z_k := z_k^v$.
Let $\Abf(X_{k})\in\Rset^{L\times L}$ and $b^{v}(X_{k})\in\Rset^{L}$ be defined as
\begin{align}
    \begin{aligned}
        \Abf(X_{k}) &= z_k(\gamma\phi(s_{k+1}) - \phi(s_k))^T,\\
        b^{v}(X_{k}) &= r_{k}^{v}z_{k}.
    \end{aligned}
    \label{notation:Akbk}  
\end{align}
Then the update of $\theta_{k}^{v}$ in \eqref{alg:theta_v} can be rewritten as
\begin{align}
    \theta_{k+1}^{v} = \sum_{u\in\Ncal_{v}}W_{vu}\theta_{k}^{u} + \alpha_{k}(\Abf(X_{k})\theta_{k}^{v} + b^{v}(X_{k}))  
    \label{sec_analysis:theta_v}
\end{align}
Let $\btheta$ be the average of $\theta^v$, $\btheta = (1/N)\sum_{v}\theta^{v}$.  Since $\Wbf$ is doubly stochastic, taking the average of Eq.\ \eqref{sec_analysis:theta_v} yields 
\begin{align}
    \btheta_{k+1} =  \btheta_{k} + \alpha_{k}\big(\Abf(X_{k})\btheta_{k} + \bbar(X_{k})\big)\label{sec_analysis:theta_bar},
\end{align}
where $\bbar(X_{k})$ is 
\begin{align}
    \bbar(X_{k}) = \frac{1}{N}\sum_{v\in\Vcal}b^{v}(X_{k}) = \frac{1}{N} \sum_{v\in\Vcal}r_{k}^{v}z_{k}\in\Rset^{L}.\label{sec_analysis:bbark}    
\end{align}
 For convenience, let $\Theta$ and $\Bbf(X_{k})$ be two matrices defined as
\begin{align}
\Theta \triangleq \left[\begin{array}{ccc}
\mbox{---}   & \left[\theta^{1}\right]^{T} & \mbox{---}\\
\cdots     & \cdots & \cdots\\
\mbox{---}   & \left[\theta^{N}\right]^{T} & \mbox{---}
\end{array}\right]\in\Rset^{N\times L},\quad\Bbf(X_{k}) \triangleq \left[\begin{array}{ccc}
\mbox{---}   & \left[b^{1}(X_{k})\right]^{T} & \mbox{---}\\
\cdots     & \cdots & \cdots\\
\mbox{---}   & \left[b^{N}(X_{k})\right]^{T} & \mbox{---}
\end{array}\right]\in\Rset^{N\times L},\label{notation:Theta_B}
\end{align}
Thus, the matrix form of Eq.\ \eqref{sec_analysis:theta_v} can be given as
\begin{align}
\Theta_{k+1} = \Wbf\Theta_{k} + \alpha_{k}\Theta_{k}\Abf(X_{k})^T + \alpha_{k}\Bbf(X_{k}).\label{sec_proofs:Theta}
\end{align}
Next, we make the following  assumptions that are fairly standard in the existing literature of consensus and reinforcement learning  \cite{Tsitsiklis1997_TD,SrikantY2019_FiniteTD,DoanMR2018b}. 
\begin{assump}\label{assump:doub_stoch}
The matrix $\Wbf$, whose $(i,j)$-th entries are $w_{uv}$, is doubly stochastic, i.e., $\sum_{u=1}^n w_{uv} =1$ for all $j$ and $\sum_{v=1}^n w_{uv} = 1$ for all $i$. Finally, $w_{uu} >0$ and the weights $w_{uv} > 0$ if and only if $(u, v) \in \Ecal$ otherwise $w_{uv} = 0$. 
\end{assump}
\begin{assump}\label{assump:reward}
All the local rewards are uniformly bounded, i.e., there exist a constant $R$, for all $v\in\Vcal$ such that $|\,\Rcal^{v}(i,j)\,|\leq R$, for all $i,j\in\Scal$. 
\end{assump}
\begin{assump}\label{assump:features}
The feature vectors $\{\phi_{\ell}\}$, for all $\ell\in\{1,\ldots,L\}$, are linearly independent, i.e., the matrix $\Phi$ has full column rank. In addition, we assume that all feature vectors $\phi(s)$ are uniformly bounded, i.e., $\|\phi(s)\|\leq 1$.  
\end{assump}
\begin{assump}\label{assump:Markov}
The Markov chain associated with $\Pbf$ is irreducible and aperiodic. 
\end{assump}
Assumption \ref{assump:doub_stoch} implies that $\Wbf$ has a largest singular value of $1$, and its other singular values are strictly less than $1$; see for example, the Perron-Frobenius theorem \cite{HJ1985}. We denote by $\sigma_2\in(0,1)$ the second largest singular value of $\Wbf$, which is a key quantity in the analysis of the mixing time of a Markov chain with transition probabilities given by $\Wbf$.




Under Assumption \ref{assump:reward} the accumulative reward $J^*$ is well defined, while under Assumption \ref{assump:features}, the projection operator $\Pi$ is well defined. If there are some dependent $\phi_{\ell}$, we can simply disregard those dependent feature vectors. Note that the uniform boundedness of $\phi_\ell$ can be guaranteed through feature normalization. Moreover, Assumption \ref{assump:Markov} 
implies that the underlying Markov chain $\Pbf$ is ergodic. First, this guarantees that there exists a unique stationary distribution $\pi$ with positive entries and 
\begin{align*}
\lim_{k\rightarrow\infty}\Eset[\Abf(X_{k})] = \Abf \text{ and } \lim_{k\rightarrow\infty}\Eset[\bbar(X_{k})] = b.    
\end{align*}

Second, it implies that the Markov chain mixes at a geometric rate \cite{Bremaud2000}. In particular, we consider the following definition of the mixing time of a Markov chain.  
\begin{definition}\label{def:mixing}
Given a positive constant $\alpha$, we denote by $\tau(\alpha)$ the mixing time of the Markov chain $\{X_{k}\}$ given as
\begin{align}
\begin{aligned}
&\|\;\Eset[\Abf(X_{k}) - \Abf\,|\,X_{0} = X]\;\| \leq \alpha,\qquad \forall\; X,\,\forall\, k\geq \tau(\alpha) \\
&\|\;\Eset[\bbar(X_{k}) - b\,|\,X_{0} = X]\;\| \leq \alpha,\qquad \forall\; X,\,\forall\, k\geq \tau(\alpha),
\end{aligned}
\label{def_mixing:tau}
\end{align}
where $b$ and $\Abf$ are given in Eqs.\ \eqref{opt_cond} and \eqref{notation:A}, respectively. In addition, under Assumption \ref{assump:Markov} the Markov chain $\{X_{k}\}$ has a geometric mixing time \cite{Bremaud2000}, i.e., there exist a constant $C$ such that given a small constant $\alpha$ we have
\begin{align}
\tau(\alpha) = C\log\left(\frac{1}{\alpha}\right).\label{notation:geo_mix}
\end{align} 
\end{definition} 
Finally, using Eq.\ \eqref{notation:Akbk} and Assumptions \ref{assump:reward} and \ref{assump:features} the induced $2$-norm of $\|\Abf(X_{k})\|$ and $\|b^{v}(X_{k})\|$ can be upper bounded by
\begin{align}
\begin{aligned}
\|\Abf(X_{k})\| &\leq\|z_k(\gamma\phi(s_{k+1}) - \phi(s_k))^T\| \leq (1+\gamma)\sum_{u=0}^{t}(\gamma\lambda)^{t-u}\leq \frac{1+\gamma}{1-\gamma\lambda}\\
\|b^{v}(X_{k})\| &\leq \|r^{v}_{k}z_{k}\| \leq  \frac{R}{1-\gamma\lambda}\cdot
\end{aligned}\label{sec_analysis:Akbk_bound} 
\end{align}
In addition, since $\lim_{k\rightarrow\infty}\Eset[\Abf(X_{k})] = \Abf$ and $\lim_{k\rightarrow\infty}\Eset[\bbar(X_{k})] = b$ we also have
\begin{align}
\|\Abf\| \leq  \frac{1+\gamma}{1-\gamma\lambda}\qquad\text{and}\qquad \|b\|\leq \frac{R}{1-\gamma\lambda},\label{sec_analysis:Ab_bound}     
\end{align}
where recall that $\Abf$ and $b$ are given in Eq.\ \eqref{opt_cond}.


\subsection{Main Results}
In this section, we provide the main results of this paper, which are the finite-time analysis for the convergence of distributed {\sf TD}$(\lambda)$. We first denote by $\delta$
a positive constant satisfying
\begin{align}
\delta = \sigma_{2} + \frac{1-\gamma\lambda}{1+\gamma}\alpha, \label{notation:delta}
\end{align}
where $\sigma_{2}\in(0,1)$ is the second largest singular value of $\Wbf$ and $\alpha$ is a constant chosen such that $\delta\in(0,1)$. Our first result is to establish the rate of convergence of the distributed {\sf TD}$(\lambda)$ when the step size is this constant $\alpha$. In particular, under some proper choice of $\alpha$, we show that the average mean-squared error across the nodes converges with a linear rate to a ball centered at the $\theta^*$. In addition, the size of this ball is a function of $\alpha$.  The following theorem states this result formally.   

\begin{theorem}\label{thm:linear}
Suppose that Assumptions \ref{assump:doub_stoch}--\ref{assump:Markov} hold. Let  $\{\theta_{k}^{v}\}$, for all $v\in\Vcal$, be generated by Algorithm \ref{alg:DTD} and $\sigma_{\min} > 0 $ be the smallest singular value of $-\Abf$. Let $\Psi_{1}$ and $\Psi_{2}$ be the constants defined
\begin{align}
\begin{aligned}
\Psi_{1} &\triangleq  4\left(36 + \frac{(229 + 42R)(1+\gamma)^2\tau(\alpha)}{(1-\gamma\lambda)^2}\right)\\ 
\Psi_{2} &\triangleq \|\theta^*\|^2\Psi_{1} + 2\left(32R^2+ 2\|\theta^*\|^2 + 1\right)\\
&\qquad +  \frac{\Big(50R^2+32(R+1)^3+100(R+\|\theta^*\|)^2\Big)(1+\gamma)^2\tau(\alpha)}{(1-\gamma\lambda)^2}\cdot
\end{aligned}\label{thm_linear:constants}
\end{align}
Let $\alpha_{k} = \alpha$ and $\tau(\alpha)$ be the corresponding mixing time defined in \eqref{def_mixing:tau}, where $\alpha$ satisfies 
\begin{align}
0 < \alpha < \min\left\{\frac{(1-\gamma\lambda)(1-\sigma_{2})}{1+\gamma}\,,\,\frac{(1-\gamma\lambda)\log(2)}{(1+\gamma)\tau(\alpha)}\;,\;\frac{\sigma_{\min}}{\Psi_{1}}\right\}.\label{thm_linear:stepsize}
\end{align}
Then we have for all $k\geq \tau(\alpha)$
\begin{align}
\frac{1}{N}\sum_{v\in\Vcal}\Eset[\|\theta_{k}^{v}-\theta^*\|^2] &\leq \frac{4\Eset[\|\Theta_{0}\|^2]}{N}\delta^{2k} + \Big(20\Eset\left[\|\btheta_{0}-\theta^*\|^2\right] + 16(\|\theta^*\| + R)^2\Big)(1-\sigma_{\min}\alpha)^{k-\tau(\alpha)}\notag\\ 
&\qquad + \frac{4R^2\alpha^2}{(1-\gamma\lambda)^2(1-\delta)^2} + \frac{2\Psi_{2}\alpha}{\sigma_{\min}},\label{thm_linear:Ineq}
\end{align}
where $C$ is given in \eqref{notation:geo_mix} and $\delta$ is defined in \eqref{notation:delta}. 
\end{theorem}

\begin{remark}
We first note that by \eqref{notation:geo_mix} we have 
\begin{align}
\limsup_{\alpha\rightarrow 0}\alpha\tau(\alpha) = 0.\label{notation:lim_alpha_tau}
\end{align}
Thus, one can choose such an $\alpha$ to satisfy Eq.\ \eqref{thm_linear:stepsize}. Second, as can be seen from our result in \eqref{thm_linear:Ineq} that the variance in our algorithm, i.e., the   quantity 
\begin{align*}
\frac{4R^2\alpha^2}{(1-\gamma\lambda)^2(1-\delta)^2}    
\end{align*}
theoretically explains the empirical observation of {\sf TD}$(\lambda)$, where $\lambda = 1$ gives the best approximation and $\lambda = {0}$ yields faster convergence under large variance. Indeed, the quantity above is smallest when $\lambda = {0}$ and gets larger as $\lambda$ increases to $1$. On the other hand, when $\lambda= 1$ one can see from \eqref{theta*:opt_bound} that $\tilde{J} = \pi J$, the best approximation of $J$ in the subspace spanned by the feature vectors.     
\end{remark}
Our second main result is to show the rate of distributed {TD}$(\lambda)$ under time-varying step sizes $\alpha_{k} = \alpha_{0}\,/\,(k+1)$ for some positive constant $\alpha_{0}$. Under this condition, we show that the average of mean square errors at each agent asymptotically converges to $0$ with a sublinear rate, which is formally stated in the following theorem.

\begin{theorem}\label{thm:sublinear}
Suppose that Assumptions \ref{assump:doub_stoch}--\ref{assump:Markov} hold. Let the sequence $\{\theta_{k}^{v}\}$, for all $v\in\Vcal$, be generated by Algorithm \ref{alg:DTD}. Let $\sigma_{\min} > 0 $ be the smallest singular value of $-\Abf$ and $\alpha_{k} = \alpha_{0}\,/\,(k+1)$ for some $\alpha_{0} \geq 1/\sigma_{\min}$. Denote by $\Psi_{3}$ and $\Psi_{4}$ two constants as 
\begin{align*}
&\Psi_{3} \triangleq  4\left(36 + \frac{(229 + 42R)(1+\gamma)^2}{(1-\gamma\lambda)^2}\right)\\ 
&\Psi_{4} \triangleq \|\theta^*\|^2\Psi_{3} + 2\left(32R^2+ 2\|\theta^*\|^2 + 1 +  \frac{\Big(50R^2+32(R+1)^3+100(R+\|\theta^*\|)^2\Big)(1+\gamma)^2}{(1-\gamma\lambda)^2}\right).
\end{align*} 
In addition, let $\Kcal^*$ be a positive integer such that for all $k\geq \Kcal^*$
\begin{align}
\begin{aligned}
&\sigma_{2} + \frac{1+\gamma}{1-\gamma\lambda}\alpha_{k} \leq \sigma_{2} + \frac{1+\gamma}{1-\gamma\lambda}\alpha =  \delta \in(0,1)\\
&\tau(\alpha_k)\alpha_{k-\tau(\alpha_k)}\leq \min\left\{\frac{(1-\gamma\lambda)\log(2)}{1+\gamma},\frac{\sigma_{\min}}{\Psi_3}\right\}
\end{aligned},\label{thm_sublinear:stepsize}    
\end{align} 
where $\alpha$ satisfies \eqref{thm_linear:stepsize} and $\delta$ is defined in \eqref{notation:delta}. Then we have for all $k\geq \Kcal^*$
\begin{align}
\frac{1}{N}\sum_{v\in\Vcal}|\Eset[\|\theta_{k}^{v}-\theta^*\|^2] 
&\leq \frac{6\Eset[\|\Theta_{\Kcal^*}\|^2]}{N}\delta^{2k-2\Kcal^*} + \frac{2\Kcal^*}{k+1}\Eset\left[\|\btheta_{\Kcal^*}-\theta^*\|^2\right] + \frac{6 R^2\alpha_{0}^2}{(1-\gamma\lambda)^2(1-\delta)^2}\delta^{k} \notag\\
&\qquad  + \frac{6R^2}{(1-\gamma\lambda)^2(1-\delta)^2}\frac{1}{(k+1)^2} + \frac{2\Psi_{4}C\alpha_{0}\log^2(\frac{k+1}{\alpha_{0}})}{k+1} \cdot\label{thm_sublinear:Ineq}
\end{align}
\end{theorem}


\section{Finite-time analysis of distributed {\sf TD}$(\lambda)$}\label{sec:proofs}
In this section, we provide the analysis for the main results in this paper, that is, the proofs of Theorems \ref{thm:linear} and \ref{thm:sublinear}. We start by first considering an upper bound for the consensus error, $\|\Theta - \1\btheta^{T}\|$,
which measures the difference between the agents' estimate $\theta_{v}$ and their average $\btheta$. The results in this lemma will be used in the analysis of both Theorems \ref{thm:linear} and \ref{thm:sublinear}, so we present here for convenience. 

\begin{lemma}\label{lem:consensus_const}
Suppose that Assumptions \ref{assump:doub_stoch}--\ref{assump:features} hold. Let the sequence $\{\theta_{k}^{v}\}$, for all $v\in\Vcal$, be generated by Algorithm \ref{alg:DTD}. We consider the following two cases

\begin{enumerate}[leftmargin = 5mm]
\item Let $\sigma_2\in(0,1)$ be the second largest singular value of $\Wbf$ and $\alpha_{k} = \alpha$ satisfy 
\begin{align}
0 \leq \alpha < \frac{(1-\sigma_{2})(1-\gamma\lambda)}{1+\gamma}, \label{consensus_const:alpha}
\end{align}
In addition, let $\delta\in(0,1)$ be defined in \eqref{notation:delta}. 
Then we have for all $k\geq 0$
\begin{align}
\|\Theta_{k}-\1\btheta_{k}^{T}\| \leq \delta^{k}\|\Theta_{0}\| + \frac{\sqrt{N}R\alpha}{(1-\gamma\lambda)(1-\delta)}\cdot\label{lem_consensus_const:Ineq}
\end{align}
\item Let $\alpha_{k} = \alpha_{0}/(k+1)$ for some positive constant $\alpha_{0}$ and $\Kcal^*_{1}$ be a positive constant such that  
\begin{align*}
\sigma_{2} + \frac{1+\gamma}{1-\gamma\lambda}\alpha_{k} \leq \sigma_{2} + \frac{1+\gamma}{1-\gamma\lambda}\alpha =  \delta \in(0,1),\qquad \forall k\geq \Kcal^*,    
\end{align*}
where $\alpha$ satisfies Eq.\ \eqref{consensus_const:alpha}.
Then we obtain for all $k\geq \Kcal^*$ that
\begin{align}
\|\Theta_{k}-\1\btheta_{k}^{T}\| &\leq \delta^{k-\Kcal^*}\|\Theta_{\Kcal^*}\| + \frac{\sqrt{N}R\alpha_{0}\delta^{k/2}}{(1-\gamma\lambda)(1-\delta)} + \frac{\sqrt{N}R\alpha_{k/2}}{(1-\gamma\lambda)(1-\delta)}\cdot\label{lem_consensus_tv:Ineq} 
\end{align}
\end{enumerate}
\end{lemma}


\subsection{Constant Step Sizes}\label{sec_proofs:const}
We now consider the case where step sizes $\alpha_k$ are constants, i.e.. $\alpha_{k} = \alpha$, for some positive constant $\alpha$.  
To show the result in Theorem \ref{thm:linear}, we first consider the following lemma about an useful error bound of the bias term  generated by Algorithm \ref{alg:DTD}. For an ease of exposition, the proof of this lemma is presented in the Appendix \ref{apx_const}.

\begin{lemma}\label{lem_const:bias}
Suppose that Assumptions \ref{assump:doub_stoch}--\ref{assump:Markov} hold. Let the sequence $\{\theta_{k}^{v}\}$, for all $v\in\Vcal$, be generated by Algorithm \ref{alg:DTD}. Let $\alpha_{k}=\alpha$ satisfying \eqref{thm_linear:stepsize}. Then for all $k\geq\tau(\alpha)$ we have
\begin{align}
&\Big|\Eset\left[(\btheta_{k}-\theta^*)^{T}(\Abf(X_{k})\btheta_{k} -\Abf\btheta_{k} + \bbar(X_{k}) - b)\,|\,\Fcal_{k-\tau(\alpha)}\right]\Big|\notag\\
& \leq \left(36 + \frac{(228 + 42R)(1+\gamma)^2\tau(\alpha)}{(1-\gamma\lambda)^2}\right)\alpha \Eset\left[\|\btheta_{k}\|^2 \,|\,\Fcal_{k-\tau(\alpha)}\right]\notag\\
&\quad + \left(32R^2+ 2\|\theta^*\|^2 + 1 +  \frac{\Big(48R^2+32(R+1)^3+100(R+\|\theta^*\|)^2\Big)(1+\gamma)^2\tau(\alpha)}{(1-\gamma\lambda)^2}\right)\alpha.\label{lem_const_bias:Ineq}    
\end{align}
\end{lemma}
We are now ready to proceed the analysis of Theorem \ref{thm:linear} as follows. 
\begin{proof}[\textbf{Proof of Theorem} \ref{thm:linear}]
Recall from Eq.\ \eqref{sec_analysis:theta_bar} with $\alpha_{k} = \alpha$ that
\begin{align*}
\btheta_{k+1} &= \btheta_{k} + \alpha\Abf(X_{k})\btheta_{k} + \alpha_{k}\bbar(X_{k}), 
\end{align*}
which gives for all $k\geq0$
\begin{align}
\|\btheta_{k+1}-\theta^*\|^2 &= \|\btheta_{k}-\theta^* +  \alpha\Abf(X_{k})\btheta_{k} + \alpha\bbar(X_{k})\|^2\notag\\
&= \|\btheta_{k}-\theta^*\|^2 + \|\alpha\Abf(X_{k})\btheta_{k} +  \alpha\bbar(X_{k})\|^2 + 2\alpha(\btheta_{k}-\theta^*)^{T}( \Abf(X_{k})\btheta_{k} + \bbar(X_{k}))\notag\\
&= \|\btheta_{k}-\theta^*\|^2 + \alpha^2\|\Abf(X_{k})\btheta_{k} + \bbar(X_{k})\|^2 + 2\alpha(\btheta_{k}-\theta^*)^{T}(\Abf\btheta_{k}+b)\notag\\
&\qquad + 2\alpha(\btheta_{k}-\theta^*)^{T}(\Abf(X_{k})\btheta_{k} + \bbar(X_{k})-\Abf\btheta_{k} - b)\notag\\
&\leq \|\btheta_{k}-\theta^*\|^2 + 2\alpha^2\|\Abf(X_{k})\btheta_{k}\|^2 + 2\alpha^2\|\bbar(X_{k})\|^2 + 2\alpha(\btheta_{k}-\theta^*)^{T}(\Abf\btheta_{k}+b)\notag\\
&\qquad + 2\alpha(\btheta_{k}-\theta^*)^{T}(\Abf(X_{k})\btheta_{k} -\Abf\btheta_{k} + \bbar(X_{k}) - b)\notag\\
& \leq \|\btheta_{k}-\theta^*\|^2 + \frac{2(1+\gamma)^2\alpha^2}{(1-\gamma\lambda)^2}\|\btheta_{k}\|^2 + \frac{2R^2\alpha^2}{(1-\gamma\lambda)^2} + 2\alpha(\btheta_{k}-\theta^*)^{T}\Abf(\btheta_{k}-\theta^*)\notag\\
&\qquad + 2\alpha\Big|(\btheta_{k}-\theta^*)^{T}(\Abf(X_{k})\btheta_{k} -\Abf\btheta_{k} + \bbar(X_{k}) - b)\Big|,\label{thm_linear:Eq1}
\end{align}
where in the last inequality we use Eq.\ \eqref{sec_analysis:Akbk_bound} and the fact that $\Abf\theta^* = -b$. Next, since $\Abf$ is negative definite and let $\sigma_{\min} > 0 $ be the smallest singular value of $-\Abf$, we have
\begin{align}
(\btheta_{k}-\theta^*)^T\Abf(\btheta_{k}-\theta^*) \leq -\sigma_{\min}\|\btheta_{k}-\theta^*\|^2.     \label{thm_linear:Eq1d}
\end{align}
We now take the expectation on both sides of Eq.\ \eqref{thm_linear:Eq1} and use Eqs.\ \eqref{lem_const_bias:Ineq} and \eqref{thm_linear:Eq1d} to have
\begin{align*}
&\Eset[\|\btheta_{k+1}-\theta^*\|^2]\notag\\
&\leq (1-2\sigma_{\min}\alpha)\Eset\left[\|\btheta_{k}-\theta^*\|^2\right]  + \frac{2(1+\gamma)^2\alpha^2}{(1-\gamma\lambda)^2}\Eset[\|\btheta_{k}\|^2] + \frac{2R^2\alpha^2}{(1-\gamma\lambda)^2}\notag\\
&\quad + 2\left(36 + \frac{(228 + 42R)(1+\gamma)^2\tau(\alpha)}{(1-\gamma\lambda)^2}\right)\alpha^2 \Eset\left[\|\btheta_{k}\|^2 \right]\notag\\
&\quad+ 2\left(32R^2+ 2\|\theta^*\|^2 + 1 +  \frac{\Big(48R^2+32(R+1)^3+100(R+\|\theta^*\|)^2\Big)(1+\gamma)^2\tau(\alpha)}{(1-\gamma\lambda)^2}\right)\alpha^2\notag\\
&\leq (1-2\sigma_{\min}\alpha)\Eset\left[\|\btheta_{k}-\theta^*\|^2\right]  + 2\left(36 + \frac{(229 + 42R)(1+\gamma)^2\tau(\alpha)}{(1-\gamma\lambda)^2}\right)\alpha^2 \Eset\left[\|\btheta_{k}\|^2 \right]\notag\\
&\qquad + 2\left(32R^2+ 2\|\theta^*\|^2 + 1 +  \frac{\Big(50R^2+32(R+1)^3+100(R+\|\theta^*\|)^2\Big)(1+\gamma)^2\tau(\alpha)}{(1-\gamma\lambda)^2}\right)\alpha^2,
\end{align*}
which by applying the Cauchy-Schwarz inequality to the second term yields  
\begin{align}
&\Eset[\|\btheta_{k+1}-\theta^*\|^2]\notag\\
&\leq (1-2\sigma_{\min}\alpha)\Eset\left[\|\btheta_{k}-\theta^*\|^2\right]+ 4\left(36 + \frac{(229 + 42R)(1+\gamma)^2\tau(\alpha)}{(1-\gamma\lambda)^2}\right)\alpha^2 \Eset\left[\|\btheta_{k}-\theta^*\|^2 \right]\notag\\
&\qquad + 4\left(36 + \frac{(229 + 42R)(1+\gamma)^2\tau(\alpha)}{(1-\gamma\lambda)^2}\right)\|\theta^*\|^2\alpha^2\notag\\
&\qquad + 2\left(32R^2+ 2\|\theta^*\|^2 + 1 +  \frac{\Big(50R^2+32(R+1)^3+100(R+\|\theta^*\|)^2\Big)(1+\gamma)^2\tau(\alpha)}{(1-\gamma\lambda)^2}\right)\alpha^2\notag\\
&= (1-2\sigma_{\min}\alpha)\Eset\left[\|\btheta_{k}-\theta^*\|^2\right] + \Psi_{1}\alpha^2 \Eset\left[\|\btheta_{k}-\theta^*\|^2 \right] + \Psi_{2}\alpha^2,\label{thm_linear:Eq2}
\end{align}
where $\Psi_{1}$ and $\Psi_{2}$ are two constants defined as 
\begin{align*}
\Psi_{1} &\triangleq  4\left(36 + \frac{(229 + 42R)(1+\gamma)^2\tau(\alpha)}{(1-\gamma\lambda)^2}\right)\\ 
\Psi_{2} &\triangleq \|\theta^*\|^2\Psi_{1} + 2\left(32R^2+ 2\|\theta^*\|^2 + 1\right)\notag\\ 
&\qquad + \frac{2\Big(50R^2+32(R+1)^3+100(R+\|\theta^*\|)^2\Big)(1+\gamma)^2\tau(\alpha)}{(1-\gamma\lambda)^2}\cdot
\end{align*}
Since $\alpha$ satisfies Eq.\ \eqref{thm_linear:stepsize} we have
$
-2\sigma_{min} + \Psi_{1}\alpha  \leq -\sigma_{min}$, which by substituting into Eq.\ \eqref{thm_linear:Eq2} yields for all $k\geq \tau(\alpha)$
\begin{align}
\Eset[\|\btheta_{k+1}-\theta^*\|^2] &\leq (1-\sigma_{\min}\alpha)\Eset\left[\|\btheta_{k}-\theta^*\|^2\right] + \Psi_{2}\alpha^2\notag\\
&\leq (1-\sigma_{\min}\alpha)^{k+1-\tau(\alpha)}\Eset\left[\|\btheta_{\tau(\alpha)}-\theta^*\|^2\right] + \Psi_{2}\alpha^2\sum_{u=\tau(\alpha)}^{k}(1-\sigma_{\min}\alpha)^{k-u}\notag\\
&\leq (1-\sigma_{\min}\alpha)^{k+1-\tau(\alpha)}\Eset\left[\|\btheta_{\tau(\alpha)}-\theta^*\|^2\right] + \frac{\Psi_{2}\alpha}{\sigma_{\min}}\cdot \label{thm_linear:Eq3a}
\end{align}
On the other hand, using Eq.\ \eqref{sec_analysis:theta_bar} with $\alpha_{k} = \alpha$ we consider 
\begin{align*}
\btheta_{k+1} - \btheta_{0} = \btheta_{k} - \btheta_{0} + \alpha\Abf(X_k)(\btheta_k-\btheta_{0}) + \alpha(\Abf(X_k)\btheta_{0}+\bbar(X_k)),
\end{align*}
which by using Eq. \eqref{sec_analysis:Akbk_bound} yields 
\begin{align*}
\|\btheta_{k+1} - \btheta_{0}\| &\leq \left(1+\frac{(1+\gamma)\alpha}{1-\gamma\lambda}\right)\|\btheta_{k} - \btheta_{0}\| + \alpha\left(\frac{(1+\gamma)}{1-\gamma\lambda}\|\btheta_{0}\| +\frac{R}{1-\gamma\lambda}\right)\notag\\
&\leq \left(\frac{(1+\gamma)\|\btheta_{0}\| + R}{1-\gamma\lambda} \right)\alpha\sum_{u=0}^{k}\left(1+\frac{(1+\gamma)\alpha}{1-\gamma\lambda}\right)^{k-u}\notag\\
&= \left(\frac{(1+\gamma)\|\btheta_{0}\| + R}{1-\gamma\lambda} \right)\alpha\left(1+\frac{(1+\gamma)\alpha}{1-\gamma\lambda}\right)^{k}\frac{1-\left(1+\frac{(1+\gamma)\alpha}{1-\gamma\lambda}\right)^{-k-1}}{1-\left(1+\frac{(1+\gamma)\alpha}{1-\gamma\lambda}\right)^{-1}}\notag\\
&\leq \left(\frac{(1+\gamma)\|\btheta_{0}\| + R}{1-\gamma\lambda} \right)\alpha\left(1+\frac{(1+\gamma)\alpha}{1-\gamma\lambda}\right)^{k+1}\frac{1-\gamma\lambda}{(1+\gamma)\alpha}\notag\\
&= \frac{(1+\gamma)\|\btheta_{0}\| + R}{1+\gamma} \left(1+\frac{(1+\gamma)\alpha}{1-\gamma\lambda}\right)^{k+1}\notag\\
&\leq \frac{(1+\gamma)\|\btheta_{0}\| + R}{1+\gamma}\exp\left\{\frac{(k+1)(1+\gamma)\alpha}{1-\gamma\lambda}\right\},
\end{align*}
where the last inequality is due to the relation $1+x\leq \exp(x)$ for any $x\geq0$. Thus, since $\alpha$ satisfies \eqref{thm_linear:stepsize}, i.e., $\frac{\alpha\tau(\alpha)(1+\gamma)}{1-\gamma\lambda} \leq \log(2)$, the preceding relation yields
\begin{align*}
\|\btheta_{\tau(\alpha)} - \btheta_{0}\| &\leq \frac{(1+\gamma)\|\btheta_{0}\| + R}{1+\gamma}\exp\left\{\frac{\alpha\tau(\alpha)(1+\gamma)}{1-\gamma\lambda}\right\}\leq \frac{2(1+\gamma)\|\btheta_{0}\| + 2R}{1+\gamma}\notag\\
&\leq 2\|\btheta_{0}-\theta^*\| +  2(\|\theta^*\| + R),
\end{align*}
which implies that
\begin{align*}
\|\btheta_{\tau(\alpha)} - \theta^*\|^2 \leq 2\|\btheta_{\tau(\alpha)} - \btheta_{0}\|^2 + 2\|\btheta_{0} - \theta^*\|^2 \leq 10\|\btheta_{0}-\theta^*\| +  8(\|\theta^*\| + R)^2.
\end{align*}
Substituting the preceding relation into Eq.\ \eqref{thm_linear:Eq3a} yields
\begin{align}
\Eset[\|\btheta_{k+1}-\theta^*\|^2] &\leq (1-\sigma_{\min}\alpha)^{k+1-\tau(\alpha)}\Big(10\Eset\left[\|\btheta_{0}-\theta^*\|^2\right] + 8(\|\theta^*\| + R)^2\Big) + \frac{\Psi_{2}\alpha}{\sigma_{\min}}\cdot \label{thm_linear:Eq3b}
\end{align}
Next, note that $\alpha$ also satisfies the condition in Eq.\ \eqref{consensus_const:alpha}, hence, by Eq.\ \eqref{lem_consensus_const:Ineq} we have
\begin{align*}
\sum_{v\in\Vcal}\|\theta_{k}^{v}-\btheta_{k}\|^2 \leq    2\delta^{2k}\|\Theta_{0}\|^2  + \frac{2NR^2\alpha^2}{(1-\gamma\lambda)^2(1-\delta)^2}, 
\end{align*}
which by using Eq.\ \eqref{thm_linear:Eq3b} gives Eq.\ \eqref{thm_linear:Ineq}, i.e.,
\begin{align*}
&\frac{1}{N}\sum_{v\in\Vcal}\Eset[\|\theta_{k}^{v}-\theta^*\|^2] \leq \frac{2}{N}\sum_{v\in\Vcal}\Big(\Eset[\|\theta_{k}^{v}-\btheta_{k}\|^2] + \Eset[\|\btheta_{k}-\theta^*\|^2] \Big)\notag\\ 
&\leq \frac{4\Eset[\|\Theta_{0}\|^2]}{N}\delta^{2k} + \Big(20\Eset\left[\|\btheta_{0}-\theta^*\|^2\right] + 16(\|\theta^*\| + R)^2\Big)(1-\sigma_{\min}\alpha)^{k-\tau(\alpha)}\notag\\ 
&\qquad + \frac{4R^2\alpha^2}{(1-\gamma\lambda)^2(1-\delta)^2} + \frac{2\Psi_{2}\alpha}{\sigma_{\min}}\cdot
\end{align*}
\end{proof}

\subsection{Time-Varying Step Sizes}
In this section, we provide the analysis of Theorem \ref{thm:sublinear}, that is, we consider the case $\alpha_{t} = \alpha_{0}\,/\,(t+1)$ for some positive constant $\alpha_{0}$. Similar to the case of constant step sizes, we consider the following important lemma about the bias term in our analysis considered later. The analysis of this lemma is presented in Appendix for completeness.  

\begin{lemma}\label{lem_tv:bias}
Suppose that Assumptions \ref{assump:doub_stoch}--\ref{assump:Markov} hold. Let the sequence $\{\theta_{k}^{v}\}$, for all $v\in\Vcal$, be generated by Algorithm \ref{alg:DTD}. Let $\alpha_{k}$ be a nonnegative and nonincreasing sequence of step sizes satisfying $\lim_{k\rightarrow\infty}\alpha_{k} = 0$, and let $\tau(\alpha_k)$ be the mixing time associated with $\alpha_k$. We denote by $\Kcal^*$ a positive integer defined in \eqref{thm_sublinear:stepsize}.
Then for all $k\geq\Kcal^*$ we have
\begin{align}
&\left|\Eset\left[(\btheta_{k}-\theta^*)^{T}(\Abf(X_{k})\btheta_{k} -\Abf\btheta_{k} + \bbar(X_{k}) - b)\,|\,\Fcal_{k-\tau(\alpha_{k})}\right]\right|\notag\\
&\qquad\leq \left(36 + \frac{(228 + 42R)(1+\gamma)^2}{(1-\gamma\lambda)^2}\right)\tau(\alpha_k)\alpha_{k-\tau(\alpha_{k})} \Eset\left[\|\btheta_{k}\|^2 \,|\,\Fcal_{k-\tau(\alpha_{k})}\right]\notag\\
&\quad\qquad +(32R^2+ 2\|\theta^*\|^2 + 1)\tau(\alpha_k)\alpha_{k-\tau(\alpha_{k})}\notag\\
&\quad\qquad +  \frac{\Big(48R^2+32(R+1)^3+100(R+\|\theta^*\|)^2\Big)(1+\gamma)^2}{(1-\gamma\lambda)^2}\tau(\alpha_k)\alpha_{k-\tau(\alpha_{k})}.\label{lem_tv_bias:Ineq}    
\end{align}
\end{lemma}
Using the preceding lemma, we now show the result stated in Theorem \ref{thm:sublinear}. The analysis is similar to the one of Theorem \ref{thm:linear}.

\begin{proof}[Proof of Theorem \ref{thm:sublinear}]
Similar to Eq.\ \eqref{thm_linear:Eq2}, by using Eq.\ \eqref{lem_tv_bias:Ineq} we obtain for all $k\geq \Kcal^*$
\begin{align*}
&\Eset[\|\btheta_{k+1}-\theta^*\|^2]\notag\\ 
&\leq (1-2\sigma_{\min}\alpha_{k})\Eset\left[\|\btheta_{k}-\theta^*\|^2\right]  + \frac{2(1+\gamma)^2\alpha_{k}^2}{(1-\gamma\lambda)^2}\Eset[\|\btheta_{k}\|^2] + \frac{2R^2\alpha_{k}^2}{(1-\gamma\lambda)^2}\notag\\
&\quad+ 2\left(36 + \frac{(228 + 42R)(1+\gamma)^2}{(1-\gamma\lambda)^2}\right)\tau(\alpha_k)\alpha_{k-\tau(\alpha_{k})}\alpha_{k} \Eset\left[\|\btheta_{k}\|^2 \,|\,\Fcal_{k-\tau(\alpha_{k})}\right]\notag\\
&\quad + 2\left(32R^2+ 2\|\theta^*\|^2 + 1\right)\tau(\alpha_k)\alpha_{k-\tau(\alpha_{k})}\alpha_{k}\notag\\
&\quad +  \frac{2\Big(48R^2+32(R+1)^3+100(R+\|\theta^*\|)^2\Big)(1+\gamma)^2}{(1-\gamma\lambda)^2}\tau(\alpha_k)\alpha_{k-\tau(\alpha_{k})}\alpha_{k},
\end{align*}
which yields
\begin{align}
&\Eset[\|\btheta_{k+1}-\theta^*\|^2]\notag\\ 
&\leq (1-2\sigma_{\min}\alpha_{k})\Eset\left[\|\btheta_{k}-\theta^*\|^2\right] \notag\\
&\quad + 2\left(36 + \frac{(229 + 42R)(1+\gamma)^2}{(1-\gamma\lambda)^2}\right)\tau(\alpha_k)\alpha_{k-\tau(\alpha_{k})}\alpha_{k} \Eset\left[\|\btheta_{k}\|^2 \right]\notag\\
&\quad + 2\left(32R^2+ 2\|\theta^*\|^2 + 1\right)\tau(\alpha_k)\alpha_{k-\tau(\alpha_{k})}\alpha_{k}\notag\\
&\quad + \frac{2\Big(50R^2+32(R+1)^3+100(R+\|\theta^*\|)^2\Big)(1+\gamma)^2}{(1-\gamma\lambda)^2}\tau(\alpha_k)\alpha_{k-\tau(\alpha_{k})}\alpha_{k}\allowdisplaybreaks\notag\\
&\leq (1-2\sigma_{\min}\alpha_{k})\Eset\left[\|\btheta_{k}-\theta^*\|^2\right] \notag\\
&\quad + 4\left(36 + \frac{(229 + 42R)(1+\gamma)^2}{(1-\gamma\lambda)^2}\right)\tau(\alpha_k)\alpha_{k-\tau(\alpha_{k})}\alpha_{k} \Eset\left[\|\btheta_{k}-\theta^*\|^2 \right]\notag\\
&\quad + 4\left(36 + \frac{(229 + 42R)(1+\gamma)^2}{(1-\gamma\lambda)^2}\right)\|\theta^*\|^2\tau(\alpha_k)\alpha_{k-\tau(\alpha_{k})}\alpha_{k}\notag\\
&\quad + 2\left(32R^2+ 2\|\theta^*\|^2 + 1\right)\tau(\alpha_k)\alpha_{k-\tau(\alpha_{k})}\alpha_{k}\notag\\
&\quad +  \frac{2\Big(50R^2+32(R+1)^3+100(R+\|\theta^*\|)^2\Big)(1+\gamma)^2}{(1-\gamma\lambda)^2}\tau(\alpha_k)\alpha_{k-\tau(\alpha_{k})}\alpha_{k}\notag\\
&= (1-2\sigma_{\min}\alpha_{k})\Eset\left[\|\btheta_{k}-\theta^*\|^2\right] + \Psi_{3}\tau(\alpha_k)\alpha_{k-\tau(\alpha_{k})}\alpha_{k} \Eset\left[\|\btheta_{k}-\theta^*\|^2 \right] + \Psi_{4}\tau(\alpha_k)\alpha_{k-\tau(\alpha_{k})}\alpha_{k},\label{thm_sublinear:Eq1}
\end{align}
where $\Psi_{3}$ and $\Psi_{4}$ are two constants defined as 
\begin{align*}
&\Psi_{3} \triangleq  4\left(36 + \frac{(229 + 42R)(1+\gamma)^2}{(1-\gamma\lambda)^2}\right)\\ 
&\Psi_{4} \triangleq \|\theta^*\|^2\Psi_{3} + 2\left(32R^2+ 2\|\theta^*\|^2 + 1 +  \frac{\Big(50R^2+32(R+1)^3+100(R+\|\theta^*\|)^2\Big)(1+\gamma)^2}{(1-\gamma\lambda)^2}\right).
\end{align*}
Recall from \eqref{thm_sublinear:stepsize} that 
\begin{align*}
-2\sigma_{\min} + \Psi_{3}\tau(\alpha_k)\alpha_{k-\tau(\alpha_k)} \leq -\sigma_{\min},
\end{align*}
which by substituting into the preceding relation yields for all $k\geq \Kcal^*$
\begin{align}
\Eset[\|\btheta_{k+1}-\theta^*\|^2] &\leq (1-\sigma_{\min}\alpha_{k})\Eset\left[\|\btheta_{k}-\theta^*\|^2\right] + \Psi_{4}\tau(\alpha_k)\alpha_{k-\tau(\alpha_{k})}\alpha_{k}\notag\\
&\leq \left(1-\frac{\sigma_{\min}\alpha_{0}}{k+1}\right)\Eset\left[\|\btheta_{k}-\theta^*\|^2\right] + \Psi_{4}\tau(\alpha_k)\alpha_{k-\tau(\alpha_{k})}\alpha_{k}\notag\\
&\leq \left(1-\frac{1}{k+1}\right)\Eset\left[\|\btheta_{k}-\theta^*\|^2\right] + \frac{\Psi_{4}\alpha_{0}^2 C\log(\frac{k+1}{\alpha_{0}})}{(k+1)(k+1-C\log(\frac{k+1}{\alpha_{0}}))}, \label{thm_sublinear:Eq1a}
\end{align}
where in the last inequality we use $\alpha_{0}\geq 1/\sigma_{\min}$ and Eq.\ \eqref{notation:geo_mix}. Iteratively updating Eq.\ \eqref{thm_sublinear:Eq1a} we have for all $k\geq \Kcal^*$ 
\begin{align}
&\Eset[\|\btheta_{k+1}-\theta^*\|^2] \notag\\
&\leq \frac{\Kcal^*}{k+1}\Eset\left[\|\btheta_{\Kcal^*}-\theta^*\|^2\right] + \Psi_{4}\alpha_{0}^2 \sum_{t=\Kcal^*}^{k}\frac{C\log(\frac{t+1}{\alpha_{0}})}{(t+1)(t+1-C\log(\frac{t+1}{\alpha_{0}}))}\prod_{u=t+1}^{k}\frac{u}{u+1}\notag\\
&\leq \frac{\Kcal^*}{k+1}\Eset\left[\|\btheta_{\Kcal^*}-\theta^*\|^2\right] + \frac{\Psi_{4}\alpha_{0}^2}{k+1} \sum_{t=\Kcal^*}^{k}\frac{C\log(\frac{t+1}{\alpha_{0}})}{(t+1-C\log(\frac{t+1}{\alpha_{0}}))}\notag\\
&\leq \frac{\Kcal^*}{k+1}\Eset\left[\|\btheta_{\Kcal^*}-\theta^*\|^2\right] + \frac{\Psi_{4}\alpha_{0}^2}{k+1} \sum_{t=\Kcal^*}^{k}\frac{2C\log(\frac{t+1}{\alpha_{0}})}{t+1}\notag\\
&\leq \frac{\Kcal^*}{k+1}\Eset\left[\|\btheta_{\Kcal^*}-\theta^*\|^2\right] + \frac{\Psi_{4}\alpha_{0}C\log^2(\frac{k+1}{\alpha_{0}})}{k+1}, \label{thm_sublinear:Eq2}
\end{align}
where the third inequality we use 
\begin{align*}
\frac{C\log(\frac{t+1}{\alpha_{0}})}{(t+1-C\log(\frac{t+1}{\alpha_{0}}))} \leq \frac{2C\log(\frac{t+1}{\alpha_{0}})}{t+1},    
\end{align*}
and the last inequality is due to the integral test
\begin{align*}
\sum_{t=\Kcal^*}^{k}\frac{\alpha_{0}\log(\frac{t+1}{\alpha_{0}})}{t+1}\leq \frac{1}{2}\log^2\left(\frac{k+1}{\alpha_{0}}\right).  
\end{align*}
Next, by Eq.\ \eqref{lem_consensus_tv:Ineq} we have for all $k\geq\Kcal^*$
\begin{align*}
\|\Theta_{k}-\1\btheta_{k}^{T}\|^2 &\leq 3\delta^{2k-2\Kcal^*}\|\Theta_{\Kcal^*}\|^2 + \frac{3N R^2\alpha_{0}^2\delta^{k}}{(1-\gamma\lambda)^2(1-\delta)^2} + \frac{3NR^2\alpha_{k/2}^2}{(1-\gamma\lambda)^2(1-\delta)^2},
\end{align*}
which by using Eq.\ \eqref{thm_sublinear:Eq2} gives Eq.\ \eqref{thm_sublinear:Ineq}, i.e., for all $k\geq\Kcal^*$
\begin{align*}
&\frac{1}{N}\sum_{v\in\Vcal}\Eset[\|\theta_{k}^{v}-\theta^*\|^2]  \leq \frac{2}{N}\sum_{v\in\Vcal}\Big(\Eset[\|\theta_{k}^{v}-\btheta_{k}\|^2]  + \Eset[\|\btheta_{k}-\theta^*\|^2] \Big)\notag\\ 
&\leq \frac{6\Eset[\|\Theta_{\Kcal^*}\|^2]}{N}\delta^{2k-2\Kcal^*} + \frac{6 R^2\alpha_{0}^2}{(1-\gamma\lambda)^2(1-\delta)^2}\delta^{k} + \frac{6R^2}{(1-\gamma\lambda)^2(1-\delta)^2}\frac{1}{(k+1)^2}\notag\\
&\qquad + \frac{2\Kcal^*}{k+1}\Eset\left[\|\btheta_{\Kcal^*}-\theta^*\|^2\right] + \frac{2\Psi_{4}C\alpha_{0}\log^2(\frac{k+1}{\alpha_{0}})}{k+1}\cdot
\end{align*}

\end{proof}

\section{Conclusion and Discussion}
In this paper, we consider a distributed consensus-based variant of the popular {\sf TD}$(\lambda)$ algorithm for estimating the value function of a given stationary policy. Our main contribution is to provide a finite-time analysis for the performance of distributed {\sf TD}$(\lambda)$, which has not been addressed in the existing literature of {\sf MARL}. Our results theoretically address some numerical observations of {\sf TD}$(\lambda)$ , that is, $\lambda=1$ gives the best approximation of the function values while $\lambda = 0$ leads to better performance when there is large variance in the algorithm.  One interesting question left from this work is the finite-time analysis when the policy is not stationary, e.g., distributed actor-critic methods. We believe that this paper establishes fundamental results that enable one to tackle these problems, which we leave for our future research.  

\bibliography{refs}
\bibliographystyle{IEEEtran}
\appendix

\section{Proof of Lemma \ref{lem:consensus_const}}
For convenience, we consider the following notation
\begin{align*}
\Qbf &= \Ibf - \frac{1}{N}\1\1^T\in\Rset^{N\times N}\\ 
\Ybf_{k} &= \Theta_{k} - \1\btheta_{k}^{T} = \left(\Ibf - \frac{1}{N}\1\1^T\right)\Theta_{k} = \Qbf\Theta_{k}.
\end{align*}
By Assumption \ref{assump:doub_stoch}, $\Wbf$ is doubly stochastic, which yields
\begin{align*}
\Wbf\Theta_{k} - \1\btheta_{k}^{T} =  \Wbf(\Theta_{k} - \1\btheta_{k}^{T})  = \Wbf\Qbf\Theta_{k}. 
\end{align*}
Thus, using Eqs.\ \eqref{sec_analysis:theta_bar} and \eqref{sec_proofs:Theta} we have
\begin{align*}
\Ybf_{k+1} &= \Qbf\Theta_{k+1} = \Theta_{k+1} - \1\btheta_{k+1}^{T}\notag\\ 
&= \Wbf\Theta_{k} + \alpha_{k}\Theta_{k}\Abf^{T}(X_{k}) + \alpha_{k}\Bbf(X_{k}) - \1\left(\btheta_{k} + \alpha_{k}\Abf(X_{k})\btheta_{k} + \alpha_{k}\bbar(X_{k})\right)^{T}\notag\\ 
&= \Wbf\Theta_{k} - \1\btheta_{k}^T + \alpha_k\big(\Theta_k - \1\btheta_{k}^T\big)\Abf^T(X_k) + \alpha_k\big(\Bbf(X_k) - \1\bbar(X_k)^T\big)\\
&= \Wbf\Qbf\Theta_{k} + \alpha_{k}\Qbf\Theta_{k}\Abf^{T}(X_{k}) + \alpha_{k}\Qbf\Bbf(X_{k})\\
&= \Wbf\Ybf_{k} + \alpha_{k}\Ybf_{k}\Abf^{T}(X_{k}) + \alpha_{k}\Qbf\Bbf(X_{k}),
\end{align*}
which by applying the $2$-norm on both sides and using the triangle inequality yields
\begin{align}
\|\Ybf_{k+1}\| &\leq \|\Wbf\Ybf_{k}\| + \alpha_{k}\|\Ybf_{k}\|\|\Abf(X_{k})\| + \alpha_{k}\|\Qbf\Bbf(X_{k})\|.\label{lem_consensus:Eq1a} 
\end{align}
First, since $\Wbf$ satisfies Assumption \ref{assump:doub_stoch},  by the Courant-Fisher theorem \cite{HJ1985} we have
\begin{align*}
\left\|\Wbf\Ybf_{k}\right\| = \left\|\Wbf\left(\Ibf-\frac{1}{n}\1\1^{T}\right)\Theta_{k}\right\| \leq \sigma_2\left\|\left(\Ibf-\frac{1}{n}\1\1^{T}\right)\Theta_{k}\right\| = \sigma_2\left\|\Ybf_{k}\right\|.
\end{align*}
Second, using the definition of $\Bbf(X_k)$ in \eqref{notation:Theta_B} and the relation in Eq.\ \eqref{sec_analysis:Akbk_bound} we obtain
\begin{align*}
\|\Abf(X_{k})\|\leq \frac{1+\gamma}{1-\gamma\lambda},\qquad \|\Qbf\Bbf(X_{k})\|\leq \frac{\sqrt{N}R}{1-\gamma\lambda}\cdot    
\end{align*}
Substituting the previous two relations into Eq.\ \eqref{lem_consensus:Eq1a} we have
\begin{align}
\|\Ybf_{k+1}\| &\leq \left(\sigma_2+\frac{1+\gamma}{1-\gamma\lambda}\alpha_k\right)\|\Ybf_{k}\|  + \frac{\sqrt{N}R}{1-\gamma\lambda}\alpha_{k}.\label{lem_consensus:Eq1} 
\end{align}
We now consider the following two cases for different choices of the step sizes $\alpha_k$. 
\begin{enumerate}[leftmargin = 5mm]
\item Let $\alpha_k = \alpha$ for some constant $\alpha$ satisfying 
Eq.\ \eqref{consensus_const:alpha}, which gives
\begin{align*}
\delta = \sigma_{2} + \frac{1+\gamma}{1-\gamma\lambda}\alpha \in (0,1).     
\end{align*}
Thus, using this relation into Eq.\ \eqref{lem_consensus:Eq1} gives Eq.\ \eqref{lem_consensus_const:Ineq}, i.e.,
\begin{align*}
\|\Ybf_{k+1}\| &\leq \left(\sigma_2 + \frac{1+\gamma}{1-\gamma\lambda}\alpha\right)\|\Ybf_{k}\| + \frac{\sqrt{N}R}{1-\gamma\lambda}\alpha =  \delta\|\Ybf_{k}\| +  \frac{\sqrt{N}C}{1-\gamma\lambda}\alpha\notag\\
&\leq \delta^{k+1}\|\Ybf_{0}\| + \frac{\sqrt{N}R\alpha}{1-\gamma\lambda}\sum_{u=0}^{k}\delta^{k-u}\leq \delta^{k+1}\|\Ybf_{0}\| + \frac{\sqrt{N}R\alpha}{(1-\gamma\lambda)(1-\delta)}\notag\\
&\leq \delta^{k+1}\|\Theta_{0}\| + \frac{\sqrt{N}R\alpha}{(1-\gamma\lambda)(1-\delta)}, 
\end{align*}
where in the last inequality we use $\|\Ybf_{0}\| = \|\Qbf\Theta_{0}\|\leq \|\Qbf\|\|\Theta_{0}\| \leq \|\Theta_{0}\|.$     
\item Let $\alpha_{k}$ be a sequence of step sizes satisfying $\alpha_{k} = \alpha_{0}/(k+1)$ for some constant $\alpha_0$. In addition, we have
\begin{align*}
\sigma_{2} + \frac{1+\gamma}{1-\gamma\lambda}\alpha_{k} \leq \sigma_{2} + \frac{1+\gamma}{1-\gamma\lambda}\alpha =  \delta \in(0,1),\qquad \forall k\geq \Kcal^*.    
\end{align*}
Next, using Eq.\ \eqref{lem_consensus:Eq1} above, we have for all $k\geq \Kcal^*$ that 
\begin{align}
\|\Ybf_{k+1}\| &\leq \left(\sigma_2 + \frac{1+\gamma}{1-\gamma\lambda}\alpha_{k}\right)\|\Ybf_{k}\| + \frac{\sqrt{N}R}{1-\gamma\lambda}\alpha_{k}\leq  \delta\|\Ybf_{k}\| +  \frac{\sqrt{N}R}{1-\gamma\lambda}\alpha_{k}\notag\\
&\leq \delta^{k+1-\Kcal^*}\|\Ybf_{\Kcal^*}\| + \frac{\sqrt{N}R}{1-\gamma\lambda}\sum_{u=\Kcal^*}^{k}\alpha_{k}\delta^{k-u}\notag\\
&\leq \delta^{k+1-\Kcal^*}\|\Theta_{\Kcal^*}\| + \frac{\sqrt{N}R}{1-\gamma\lambda}\left(\sum_{u=0}^{\left\lfloor  \frac{k}{2} \right\rfloor}\alpha_{k}\delta^{k-u} + \sum_{u=\left\lceil  \frac{k}{2} \right\rceil}^{k}\alpha_{k}\delta^{k-u}\right)\notag\\
&\leq \delta^{k+1-\Kcal^*}\|\Theta_{\Kcal^*}\| + \frac{\sqrt{N}R\alpha_{0}}{(1-\gamma\lambda)(1-\delta)}\delta^{(k+1)/2} + \frac{\sqrt{N}R}{(1-\gamma\lambda)(1-\delta)}\alpha_{k/2}.\label{lem_consensus:Eq2} 
\end{align}
\end{enumerate}

\section{Constant step sizes}\label{apx_const}

In this section, we provide a full analysis of Lemma \ref{lem_const:bias}. To to that, we first study a few upper bounds of the quantity $\|\btheta_{t}-\btheta_{t-\tau(\alpha)}\|^2$, which is the change of $\btheta_{t}$ in any $\tau(\alpha)$ mixing time interval. Our main observation states that under some proper choice of the step size $\alpha$, such difference should not be too large, as stated in the following lemma. In addition, these bounds will play a crucial role for our result given later.   

\begin{lemma}\label{lem_const:xbar_bound}
Suppose that Assumptions \ref{assump:doub_stoch}--\ref{assump:features} hold. Let the sequence $\{\theta_{k}^{v}\}$, for all $v\in\Vcal$, be generated by Algorithm \ref{alg:DTD}. In addition, let $\alpha_{k}=\alpha$ and $\tau(\alpha)$ be the corresponding mixing time in \eqref{def_mixing:tau} satisfying
\begin{align}
0 \leq \alpha\tau(\alpha) \leq \frac{(1-\gamma\lambda)\log(2)}{1+\gamma}, \label{const_stepsize:alpha}
\end{align}
Then for all $k\geq\tau(\alpha)$ we have
\begin{align}
&a)\qquad  \|\btheta_{k}-\btheta_{k-\tau(\alpha)}\|\leq \frac{2(1+\gamma)\tau(\alpha)}{1-\gamma\lambda}\alpha\|\btheta_{k-\tau(\alpha)}\|  + \frac{2R\alpha\tau(\alpha)}{1-\gamma\lambda}\cdot\label{lem_xbar_bound_const:Ineq1}\\
&b)\qquad \|\btheta_{k}-\btheta_{k-\tau(\alpha)}\|\leq \frac{6(1+\gamma)\alpha\tau(\alpha)}{1-\gamma\lambda}\|\btheta_{k}\|  + \frac{6R\alpha\tau(\alpha)}{1-\gamma\lambda}.\label{lem_xbar_bound_const:Ineq2}\\
&c)\qquad \|\btheta_{k}-\btheta_{k-\tau(\alpha)}\|^2 \leq \frac{72(1+\gamma)^2\alpha^2\tau^2(\alpha)}{(1-\gamma\lambda)^2}\|\btheta_{k}\|^2  + \frac{72R^2\alpha^2\tau^2(\alpha)}{(1-\gamma\lambda)^2}\leq 8\|\btheta_{k}\|^2  + 8R^2.\label{lem_xbar_bound_const:Ineq3}
\end{align}
\end{lemma}

\begin{proof}
Using Eq.\ \eqref{sec_analysis:theta_bar} with $\alpha_{k} = \alpha$ we have
\begin{align*}
\btheta_{k+1} 
&= \btheta_{k} + \alpha\Abf(X_{k})\btheta_{k} + \alpha\bbar(X_{k}),
\end{align*}
which when taking the $2$-norm on both sides and using Eq.\ \eqref{sec_analysis:Akbk_bound} yields
\begin{align}
\|\btheta_{k+1}\| &\leq \|\btheta_{k}\| + \frac{1+\gamma}{1-\gamma\lambda}\alpha\|\btheta_{k}\| + \frac{R\alpha}{1-\gamma\lambda} = \left(1 + \frac{1+\gamma}{1-\gamma\lambda}\alpha\right)\|\btheta_{k}\| + \frac{R\alpha}{1-\gamma\lambda}.\label{lem_xbar_bound_const:Eq1a}
\end{align}
First, we consider when $k\geq \tau(\alpha)$, where $\tau(\alpha)$ is the mixing time associated with $\alpha$ defined in \eqref{def_mixing:tau}. Indeed, using Eq.\ \eqref{lem_xbar_bound_const:Eq1a} we have for all $u\in[k-\tau(\alpha),k]$ and $k\geq \tau(\alpha)$ 
\begin{align}
\|\btheta_{u}\| &\leq \left(1 + \frac{1+\gamma}{1-\gamma\lambda}\alpha\right)^{u-k+\tau_{\alpha}}\|\btheta_{k-\tau(\alpha)}\| + \frac{R\alpha}{1-\gamma\lambda}\sum_{\ell=k-\tau(\alpha)}^{u-1}\left(1 + \frac{1+\gamma}{1-\gamma\lambda}\alpha\right)^{u-1-\ell}\cdot\label{lem_xbar_bound_const:Eq1b}
\end{align}
Recall from Eq. \eqref{const_stepsize:alpha} that $\alpha$ satisfies 
\begin{align}
0 \leq \alpha\tau(\alpha) \leq \frac{(1-\gamma\lambda)\log(2)}{1+\gamma}, 
\end{align}
which by using the relation $1 + x \leq e^{x}$ for all $x\geq 0$ we have for all $u\in[k-\tau(\alpha),k]$
\begin{align*}
\left(1 + \frac{1+\gamma}{1-\gamma\lambda}\alpha\right)^{u-k+\tau(\alpha)} &\leq  \left(1 + \frac{1+\gamma}{1-\gamma\lambda}\alpha\right)^{\tau(\alpha)} \leq \exp\left\{\frac{1+\gamma}{1-\gamma\lambda}\alpha\tau(\alpha)\right\} \leq 2.
\end{align*}
Using the preceding relation into Eq.\ \eqref{lem_xbar_bound_const:Eq1b} we have for all $u\in[k-\tau(\alpha),k]$ for $u\geq\tau(\alpha)$
\begin{align}
\|\btheta_{u}\| &\leq 2\|\btheta_{k-\tau(\alpha)}\| + \frac{R\alpha}{1-\gamma\lambda}\sum_{k-\tau(\alpha)}^{u-1}\left(1+\frac{1+\gamma}{1-\gamma\lambda}\alpha\right)^{u-k+\tau(\alpha)} \leq 2\|\btheta_{k- \tau(\alpha)}\| + \frac{2R\tau(\alpha)\alpha}{1-\gamma\lambda}\cdot\label{lem_xbar_bound_const:Eq1}
\end{align}
Thus, using Eq.\ \eqref{lem_xbar_bound_const:Eq1} we obtain Eq.\ \eqref{lem_xbar_bound_const:Ineq1} for all $k\geq\tau(\alpha)$
, i.e.,
\begin{align*}
\|\btheta_{k}-\btheta_{k-\tau(\alpha)}\|&\leq \sum_{u=k-\tau(\alpha)}^{k-1}\|\btheta_{u+1}-\btheta_{u}\| \stackrel{\eqref{lem_xbar_bound_const:Eq1a}}{\leq} \frac{1+\gamma}{1-\gamma\lambda}\alpha\sum_{u=k-\tau(\alpha)}^{k-1}\|\btheta_{u}\| + \frac{R\alpha\tau(\alpha)}{1-\gamma\lambda}\notag\\
&\leq \frac{1+\gamma}{1-\gamma\lambda}\alpha\sum_{u = k-\tau(\alpha)}^{k-1}\left(2\|\btheta_{k - \tau(\alpha)}\| + \frac{2R\tau(\alpha)\alpha}{1-\gamma\lambda}\right) + \frac{R\alpha\tau(\alpha)}{1-\gamma\lambda}\notag\\
&= \frac{2(1+\gamma)\alpha\tau(\alpha)}{1-\gamma\lambda}\|\btheta_{k-\tau(\alpha)}\| + \frac{(1+\gamma)\alpha\tau(\alpha)}{1-\gamma\lambda}\frac{2R\alpha\tau(\alpha)}{1-\gamma\lambda} + \frac{R\alpha\tau(\alpha)}{1-\gamma\lambda}\notag\\
&\leq \frac{2(1+\gamma)\tau(\alpha)}{1-\gamma\lambda}\alpha\|\btheta_{k-\tau(\alpha)}\|  + \frac{2R\alpha\tau(\alpha)}{1-\gamma\lambda},
\end{align*}
where in the last inequality we use Eq.\ \eqref{const_stepsize:alpha}, i.e.,
\begin{align*}
\frac{(1+\gamma)\alpha\tau(\alpha)}{1-\gamma\lambda}\leq \log(2). 
\end{align*}
Next, using the equation above we have for all $k\geq\tau(\alpha)$
\begin{align*}
\|\btheta_{k}-\btheta_{k-\tau(\alpha)}\| &\leq \frac{2(1+\gamma)\alpha\tau(\alpha)}{1-\gamma\lambda}\|\btheta_{k-\tau(\alpha)}\|  + \frac{2R\alpha\tau(\alpha)}{1-\gamma\lambda}\notag\\ 
&\leq \frac{2(1+\gamma)\alpha\tau(\alpha)}{1-\gamma\lambda}\|\btheta_{k} - \btheta_{k-\tau(\alpha)}\| + \frac{2(1+\gamma)\alpha\tau(\alpha)}{1-\gamma\lambda}\|\btheta_{k}\|  + \frac{2R\alpha\tau(\alpha)}{1-\gamma\lambda}\notag\\
&\leq \frac{2}{3}\|\btheta_{k-\tau(\alpha)}-\theta_{k}\| + \frac{2(1+\gamma)\alpha\tau(\alpha)}{1-\gamma\lambda}\|\btheta_{k}\|  + \frac{2R\alpha\tau(\alpha)}{1-\gamma\lambda},
\end{align*}
which gives Eq.\ \eqref{lem_xbar_bound_const:Ineq2}, i.e.,
\begin{align*}
\|\btheta_{k}-\btheta_{k-\tau(\alpha)}\|\leq \frac{6(1+\gamma)\alpha\tau(\alpha)}{1-\gamma\lambda}\|\btheta_{k}\|  + \frac{6R\alpha\tau(\alpha)}{1-\gamma\lambda}.
\end{align*}
Using the inequality $(x+y)^2 \leq 2x^2 + 2y^2$ for all $x,y\in\Rset$, the preceding equation also gives Eq.\ \eqref{lem_xbar_bound_const:Ineq3}, i.e., 
\begin{align*}
\|\btheta_{k}-\btheta_{k-\tau(\alpha)}\|^2 &\leq \frac{72(1+\gamma)^2\alpha^2\tau^2(\alpha)}{(1-\gamma\lambda)^2}\|\btheta_{k}\|^2  + \frac{72R^2\alpha^2\tau^2(\alpha)}{(1-\gamma\lambda)^2}\leq 8\|\btheta_{k}\|^2  + 8R^2,
\end{align*}
where in the last inequality we use 
$(1+\gamma)\alpha\tau(\alpha)\leq (1-\gamma\lambda)\log(2).$ 
\end{proof}

Next, using the previous results we now consider the following two important lemmas. We denote by $\Fcal_{k} = \{\theta_{0},\ldots,\theta_{k}\}$ the filtration containing all the information generated by Algorithm \ref{alg:DTD} up to time $k$.
\begin{lemma}\label{lem_mixing:bound}
Suppose that Assumptions \ref{assump:doub_stoch}--\ref{assump:Markov} hold. Let the sequence $\{\theta_{k}^{v}\}$, for all $v\in\Vcal$, be generated by Algorithm \ref{alg:DTD}. In addition, let $\alpha_{t}=\alpha$ where $\alpha$ satisfies \eqref{const_stepsize:alpha}. Then for all $k\geq\tau(\alpha)$ we have
\begin{align}
&\left|\Eset\left[(\btheta_{k-\tau(\alpha)}-\theta^*)^{T}(\Abf(X_{k})-\Abf)\btheta_{k-\tau(\alpha)}|\Fcal_{k-\tau(\alpha)}\right]\right|\! \leq\! 27\alpha\Eset\left[\|\btheta_{k}\|^2\,|\,\Fcal_{k-\tau(\alpha)}\right] + (24R^2\! +\! \|\theta^*\|^2)\alpha.\label{lem_mixing_bound:Ineq1a}\\
&\left|\Eset\left[(\btheta_{k-\tau(\alpha)}-\theta^*)^{T}(\bbar(X_{k}) - b)\,|\,\Fcal_{k-\tau(\alpha)}\right]\right|
\leq 9\alpha\Eset[\|\btheta_{k}\|^2\,|\,\Fcal_{k-\tau(\alpha)}] + \left(8R^2 + 1 + \|\theta^*\|\right)\alpha.\label{lem_mixing_bound:Ineq1b}
\end{align}
\end{lemma}

\begin{proof}
First, using the mixing time defined in \eqref{def_mixing:tau} we consider for all $k\geq\tau(\alpha)$
\begin{align*}
&\left|\Eset\left[(\btheta_{k-\tau(\alpha)}-\theta^*)^{T}(\Abf(X_{k})-\Abf)\btheta_{k-\tau(\alpha)}\,|\,\Fcal_{k-\tau(\alpha)}\right]\right|\notag\\ 
&\quad =  \left|(\btheta_{k-\tau(\alpha)}-\theta^*)^{T}\Eset\left[\Abf(X_{k})-\Abf\,|\,\Fcal_{k-\tau(\alpha)}\right]\btheta_{k-\tau(\alpha)}\right|\notag\\ 
&\quad \stackrel{\eqref{def_mixing:tau}}{\leq} \alpha\Eset\left[\|\btheta_{k-\tau(\alpha)}-\theta^*\|\,\|\btheta_{k-\tau(\alpha)}\|\,|\,\Fcal_{k-\tau(\alpha)}\right]\notag\\
&\quad \leq \frac{\alpha}{2}\Eset\left[\|\btheta_{k-\tau(\alpha)}\|^2 + \|\btheta_{k-\tau(\alpha)}-\theta^*\|^2\,|\,\Fcal_{k-\tau(\alpha)}\right]\notag\\ 
&\quad \leq \frac{3\alpha}{2}\Eset\left[\|\btheta_{k-\tau(\alpha)}\|^2\,|\,\Fcal_{k-\tau(\alpha)}\right] + \|\theta^*\|^2\alpha\notag\\
&\quad \leq 3\alpha\Eset\left[\|\btheta_{k}-\btheta_{k-\tau(\alpha)}\|^2\,|\,\Fcal_{k-\tau(\alpha)}\right] + 3\alpha\Eset\left[\|\btheta_{k}\|^2\,|\,\Fcal_{k-\tau(\alpha)}\right] +  \|\theta^*\|^2\alpha,
\end{align*}
which by using Eq.\ \eqref{lem_xbar_bound_const:Ineq3} yields Eq.\   \eqref{lem_mixing_bound:Ineq1a}, i.e.,
\begin{align*}
&\left|\Eset\left[(\btheta_{k-\tau(\alpha)}-\theta^*)^{T}(\Abf(X_{k})-\Abf)\btheta_{k-\tau(\alpha)}\,|\,\Fcal_{k-\tau(\alpha)}\right]\right|\notag\\
&\quad \leq 3\alpha\left(8\|\btheta_{k}\|^2  + 8R^2\right) + 3\alpha\Eset\left[\|\btheta_{k}\|^2\,|\,\Fcal_{k-\tau(\alpha)}\right] +  \|\theta^*\|^2\alpha\notag\\
&\quad \leq 27\alpha\Eset\left[\|\btheta_{k}\|^2\,|\,\Fcal_{k-\tau(\alpha)}\right] + (24R^2 + \|\theta^*\|^2)\alpha.
\end{align*}
Next, we use the definition of mixing time in Eq.\ \eqref{def_mixing:tau} again to have
\begin{align*}
&\left|\Eset\left[(\btheta_{k-\tau(\alpha)}-\theta^*)^{T}(\bbar(X_{k}) - b)\,|\,\Fcal_{k-\tau(\alpha)}\right]\right|\notag\\ 
&\qquad = \left|(\btheta_{k-\tau(\alpha)}-\theta^*)^{T}\Eset\left[\bbar(X_{k}) - b\,|\,\Fcal_{k-\tau(\alpha)}\right]\right|\notag\\
&\qquad \leq \alpha\Eset[\|\btheta_{k-\tau(\alpha)}-\theta^*\|\,|\,\Fcal_{k-\tau(\alpha)}] \leq \alpha\Eset[\|\btheta_{k-\tau(\alpha)}\|\,|\,\Fcal_{k-\tau(\alpha)}] + \|\theta^*\|\alpha\notag\\
&\qquad \leq \frac{\alpha}{2}\Eset[\|\btheta_{k-\tau(\alpha)}\|^2\,|\,\Fcal_{k-\tau(\alpha)}] + \left(\frac{1}{2} + \|\theta^*\|\right)\alpha\notag\\
&\qquad \leq \alpha\Eset[\|\btheta_{k} - \btheta_{k-\tau(\alpha)}\|^2\,|\,\Fcal_{k-\tau(\alpha)}] + \alpha\Eset[\|\btheta_{k}\|^2\,|\,\Fcal_{k-\tau(\alpha)}] + \left(\frac{1}{2} + \|\theta^*\|\right)\alpha,
\end{align*}
which by using Eq.\ \eqref{lem_xbar_bound_const:Ineq3} yields Eq.\ \eqref{lem_mixing_bound:Ineq1b}, i.e.,
\begin{align*}
&\left|\Eset\left[(\btheta_{k-\tau(\alpha)}-\theta^*)^{T}(\bbar(X_{k}) - b)\,|\,\Fcal_{k-\tau(\alpha)}\right]\right|\notag\\ 
&\qquad \leq \alpha\left(8\Eset[\|\btheta_{k}\|^2\,|\,\Fcal_{k-\tau(\alpha)}] + 8R^2\right) +  \alpha\Eset[\|\btheta_{k}\|^2\,|\,\Fcal_{k-\tau(\alpha)}] + \left(\frac{1}{2} + \|\theta^*\|\right)\alpha\notag\\
&\qquad \leq 9\alpha\Eset[\|\btheta_{k}\|^2\,|\,\Fcal_{k-\tau(\alpha)}] + \left(8R^2 + 1 + \|\theta^*\|\right)\alpha.
\end{align*}
\end{proof}

\begin{lemma}\label{lem_const:opt_bound}
Suppose that Assumptions \ref{assump:doub_stoch}--\ref{assump:Markov} hold. Let the sequence $\{\theta_{k}^{v}\}$, for all $v\in\Vcal$, be generated by Algorithm \ref{alg:DTD}. In addition, let $\alpha_{k}=\alpha$ where $\alpha$ satisfies \eqref{const_stepsize:alpha}. Then for all $k\geq\tau(\alpha)$ we have
\begin{align}
&\left |\Eset\left[(\btheta_{k-\tau(\alpha)}-\theta^*)^{T}(\Abf(X_{k})-\Abf)(\btheta_{k}-\btheta_{k-\tau(\alpha)})\,|\,\Fcal_{k-\tau(\alpha)}\right]\right|\notag\\
&\qquad\qquad \leq \frac{108(1+\gamma)^2\alpha\tau(\alpha)}{(1-\gamma\lambda)^2}\Eset\left[\|\btheta_{k}\|^2\,|\,\Fcal_{k-\tau(\alpha)}\right] +  \frac{100(1+\gamma)^2(R+\|\theta^*\|)^2\alpha\tau(\alpha)}{(1-\gamma\lambda)^2}\cdot\label{lem_opt_bound_const:Ineq1}\\
&\left|\Eset\left[(\btheta_{k}-\btheta_{k-\tau(\alpha)})^{T}(\Abf(X_{k}) -\Abf)\btheta_{k-\tau(\alpha)}\,|\,\Fcal_{k-\tau(\alpha)}\right]\right|\notag\\
&\qquad\qquad \leq \frac{36(R+2)(1+\gamma)^2\alpha\tau(\alpha)}{(1-\gamma\lambda)^2}\Eset\left[\|\btheta_{k}\|^2\,|\,\Fcal_{k-\tau(\alpha)}\right] + \frac{32R(R+1)^2(1+\gamma)^2\alpha\tau(\alpha)}{(1-\gamma\lambda)^2}\cdot\label{lem_opt_bound_const:Ineq2}\\
&\left|\Eset\left[(\btheta_{k}-\btheta_{k-\tau(\alpha)})^{T}(\Abf(X_{k}) -\Abf)(\btheta_{k}-\btheta_{k-\tau(\alpha)})\,|\,\Fcal_{k-\tau(\alpha)}\right]\right|\notag\\
&\qquad\qquad \leq \frac{48(1+\gamma)^2\alpha\tau(\alpha)}{(1-\gamma\lambda)^2} \Eset\left[\|\btheta_{k}\|^2 \,|\,\Fcal_{k-\tau(\alpha)}\right] + \frac{48R^2\alpha\tau(\alpha)}{(1-\gamma\lambda)^2}\cdot\label{lem_opt_bound_const:Ineq3}\\
&\left|\Eset\left[(\btheta_{k}-\btheta_{k-\tau(\alpha)})^{T}(\bbar(X_{k}) - b)\,|\,\Fcal_{k-\tau(\alpha)}\right]\right|\notag\\
&\qquad\qquad \leq \frac{6R(1+\gamma)\alpha\tau(\alpha)}{(1-\gamma\lambda)^2}\Eset[\|\theta_{k}\|^2\,|\,\Fcal_{k-\tau(\alpha)}] + \frac{12R(R+1)\alpha\tau(\alpha)}{(1-\gamma\lambda)^2}\cdot\label{lem_opt_bound_const:Ineq4}
\end{align}
\end{lemma}

\begin{proof}
First, we show Eq.\ \eqref{lem_opt_bound_const:Ineq1}. using Eqs.\ \eqref{sec_analysis:Akbk_bound} and \eqref{sec_analysis:Ab_bound} we have for all $k\geq \tau(\alpha)$
\begin{align*}
&\left |\Eset\left[(\btheta_{k-\tau(\alpha)}-\theta^*)^{T}(\Abf(X_{k})-\Abf)(\btheta_{k}-\btheta_{k-\tau(\alpha)})\,|\,\Fcal_{k-\tau(\alpha)}\right]\right|\notag\\
&\quad \leq \Eset\left[(\|\Abf(X_{k})\|+\|\Abf\|)\|\btheta_{k-\tau(\alpha)}-\theta^*\|\|\btheta_{k}-\btheta_{k-\tau(\alpha)}\|\,|\,\Fcal_{k-\tau(\alpha)}\right]\notag\\
&\quad \leq \frac{2(1+\gamma)}{1-\gamma\lambda}\Eset\left[\|\btheta_{k-\tau(\alpha)}\|\|\btheta_{k}-\btheta_{k-\tau(\alpha)}\|\,|\,\Fcal_{k-\tau(\alpha)}\right]\notag\\ 
&\qquad + \frac{2\|\theta^*\|(1+\gamma)}{1-\gamma\lambda}\Eset\left[\|\btheta_{k}-\btheta_{k-\tau(\alpha)}\|\,|\,\Fcal_{k-\tau(\alpha)}\right],
\end{align*}
which by using Eq.\ \eqref{lem_xbar_bound_const:Ineq1} to upper bound the term $\|\btheta_{k}-\btheta_{k-\tau(\alpha)}\|$ yields
\begin{align*}
&\left|\Eset\left[(\btheta_{k-\tau(\alpha)}-\theta^*)^{T}(\Abf(X_{k})-\Abf)(\btheta_{k}-\btheta_{k-\tau(\alpha)})\,|\,\Fcal_{k-\tau(\alpha)}\right]\right|\notag\\
&\quad \leq \frac{2(1+\gamma)}{1-\gamma\lambda}\left(\frac{2(1+\gamma)\alpha\tau(\alpha)}{1-\gamma\lambda}\Eset\left[\|\btheta_{k-\tau(\alpha)}\|^2\,|\,\Fcal_{k-\tau(\alpha)}\right] + \frac{2R\alpha\tau(\alpha)}{1-\gamma\lambda}\Eset\left[\|\btheta_{k-\tau(\alpha)}\|\,|\,\Fcal_{k-\tau(\alpha)}\right]\right)\notag\\ 
&\quad \qquad + \frac{2\|\theta^*\|(1+\gamma)}{1-\gamma\lambda}\left(\frac{2(1+\gamma)\alpha\tau(\alpha)}{1-\gamma\lambda}\Eset\left[\|\btheta_{k-\tau(\alpha)}\|\,|\,\Fcal_{k-\tau(\alpha)}\right] + \frac{2R\alpha\tau(\alpha)}{1-\gamma\lambda}\right)\notag\\
&\quad= \frac{4(1+\gamma)^2\alpha\tau(\alpha)}{(1-\gamma\lambda)^2}\Eset\left[\|\btheta_{k-\tau(\alpha)}\|^2\,|\,\Fcal_{k-\tau(\alpha)}\right]\notag\\ 
&\quad \qquad  + \frac{4(1+\gamma)^2(R+\|\theta^*\|)\alpha\tau(\alpha)}{(1-\gamma\lambda)^2}\Eset\left[\|\btheta_{k-\tau(\alpha)}\|\,|\,\Fcal_{k-\tau(\alpha)}\right] + \frac{4\|\theta^*\|R(1+\gamma)\alpha\tau(\alpha)}{(1-\gamma\lambda)^2}\cdot
\end{align*}
Using the relation $2xy\leq x^2 + y^2$ to the second term on the right-hand side of above yields
\begin{align*}
&\left|\Eset\left[(\btheta_{k-\tau(\alpha)}-\theta^*)^{T}(\Abf(X_{k})-\Abf)(\btheta_{k}-\btheta_{k-\tau(\alpha)})\,|\,\Fcal_{k-\tau(\alpha)}\right]\right|\notag\\
&\quad \leq \frac{6(1+\gamma)^2\alpha\tau(\alpha)}{(1-\gamma\lambda)^2}\Eset\left[\|\btheta_{k-\tau(\alpha)}\|^2\,|\,\Fcal_{k-\tau(\alpha)}\right] +  \frac{2(1+\gamma)^2(R+\|\theta^*\|)^2\alpha\tau(\alpha)}{(1-\gamma\lambda)^2}\notag\\ 
&\quad\qquad + \frac{4R(1+\gamma)\|\theta^*\|\alpha\tau(\alpha)}{(1-\gamma\lambda)^2}\notag\\
&\quad \leq \frac{6(1+\gamma)^2\alpha\tau(\alpha)}{(1-\gamma\lambda)^2}\Eset\left[\|\btheta_{k-\tau(\alpha)}\|^2\,|\,\Fcal_{k-\tau(\alpha)}\right] +  \frac{4(1+\gamma)^2(R+\|\theta^*\|)^2\alpha\tau(\alpha)}{(1-\gamma\lambda)^2}\notag\\
&\quad \leq \frac{12(1+\gamma)^2\alpha\tau(\alpha)}{(1-\gamma\lambda)^2}\left(\Eset\left[\|\btheta_{k}-\btheta_{k-\tau(\alpha)}\|^2\,|\,\Fcal_{k-\tau(\alpha)}\right] + \Eset\left[\|\btheta_{k}\|^2\,|\,\Fcal_{k-\tau(\alpha)}\right]\right)\notag\\ 
&\quad\qquad +  \frac{4(1+\gamma)^2(R+\|\theta^*\|)^2\alpha\tau(\alpha)}{(1-\gamma\lambda)^2} ,
\end{align*}
which by using Eq.\ \eqref{lem_xbar_bound_const:Ineq3} to the first term on the right-hand side we obtain
\begin{align}
&\left |\Eset\left[(\btheta_{k-\tau(\alpha)}-\theta^*)^{T}(\Abf(X_{k})-\Abf)(\btheta_{k}-\btheta_{k-\tau(\alpha)})\,|\,\Fcal_{k-\tau(\alpha)}\right]\right|\notag\\
&\quad \leq \frac{12(1+\gamma)^2\alpha\tau(\alpha)}{(1-\gamma\lambda)^2}\left(8\Eset\left[\|\btheta_{k}\|^2\,|\,\Fcal_{k-\tau(\alpha)}\right] + 8R^2 \right)\notag\\ 
&\quad\qquad + \frac{12(1+\gamma)^2\alpha\tau(\alpha)}{(1-\gamma\lambda)^2}\Eset\left[\|\btheta_{k}\|^2\,|\,\Fcal_{k-\tau(\alpha)}\right] +  \frac{4(1+\gamma)^2(R+\|\theta^*\|)^2\alpha\tau(\alpha)}{(1-\gamma\lambda)^2} \notag\\
&\quad \leq \frac{108(1+\gamma)^2\alpha\tau(\alpha)}{(1-\gamma\lambda)^2}\Eset\left[\|\btheta_{k}\|^2\,|\,\Fcal_{k-\tau(\alpha)}\right] +  \frac{100(1+\gamma)^2(R+\|\theta^*\|)^2\alpha\tau(\alpha)}{(1-\gamma\lambda)^2}\cdot
\end{align}
Next, using Eqs.\ \eqref{sec_analysis:Akbk_bound} and  \eqref{sec_analysis:Ab_bound} again we have
\begin{align*}
&\left|\Eset\left[(\btheta_{k}-\btheta_{k-\tau(\alpha)})^{T}(\Abf(X_{k}) -\Abf)\btheta_{k-\tau(\alpha)}\,|\,\Fcal_{k-\tau(\alpha)}\right]\right|\notag\\
&\quad \leq  \frac{2(1+\gamma)}{1-\gamma\lambda}\Eset\left[\|\btheta_{k-\tau(\alpha)}\|\|\btheta_{k}-\btheta_{k-\tau(\alpha)}\|\,|\,\Fcal_{k-\tau(\alpha)}\right]
\end{align*}
which by using \eqref{lem_xbar_bound_const:Ineq1} and $2x\leq x^2 + 1$ yields
\begin{align*}
&\left|\Eset\left[(\btheta_{k}-\btheta_{k-\tau(\alpha)})^{T}(\Abf(X_{k}) -\Abf)\btheta_{k-\tau(\alpha)}\,|\,\Fcal_{k-\tau(\alpha)}\right]\right|\notag\\
&\quad\leq \frac{2(1+\gamma)}{1-\gamma\lambda}\Eset\left[\frac{2(1+\gamma)\alpha\tau(\alpha)}{1-\gamma\lambda}\|\btheta_{k-\tau(\alpha)}\|^2  + \frac{2R\alpha\tau(\alpha)}{1-\gamma\lambda}\|\btheta_{k-\tau(\alpha)}\|\,\Bigg|\,\Fcal_{k-\tau(\alpha)}\right]\\
&\quad\leq \frac{2(1+\gamma)}{1-\gamma\lambda}\Eset\left[\frac{(R+2)(1+\gamma)\alpha\tau(\alpha)}{1-\gamma\lambda}\|\btheta_{k-\tau(\alpha)}\|^2  + \frac{R\alpha\tau(\alpha)}{1-\gamma\lambda}\,\Bigg|\,\Fcal_{k-\tau(\alpha)}\right]\\
&\quad\leq  \frac{4(R+2)(1+\gamma)^2\alpha\tau(\alpha)}{(1-\gamma\lambda)^2}\Eset\left[\|\btheta_{k}-\btheta_{k-\tau(\alpha)}\|^2  + \|\btheta_{k}\|^2\,\big|\,\Fcal_{k-\tau(\alpha)}\right] + \frac{2R(1+\gamma)\alpha\tau(\alpha)}{(1-\gamma\lambda)^2}.
\end{align*}
Using Eq.\ \eqref{lem_xbar_bound_const:Ineq3} to the preceding equation we have
\begin{align*}
&\left|\Eset\left[(\btheta_{k}-\btheta_{k-\tau(\alpha)})^{T}(\Abf(X_{k}) -\Abf)\btheta_{k-\tau(\alpha)}\,|\,\Fcal_{k-\tau(\alpha)}\right]\right|\notag\\
&\quad\leq \frac{4(R+2)(1+\gamma)^2\alpha\tau(\alpha)}{(1-\gamma\lambda)^2}\Eset\left[9\|\btheta_{k}\|^2 + 8R^2\,|\,\Fcal_{k-\tau(\alpha)}\right] + \frac{2R(1+\gamma)\alpha\tau(\alpha)}{(1-\gamma\lambda)^2}\\
&\quad \leq \frac{36(R+2)(1+\gamma)^2\alpha\tau(\alpha)}{(1-\gamma\lambda)^2}\Eset\left[\|\btheta_{k}\|^2\,|\,\Fcal_{k-\tau(\alpha)}\right] + \frac{32R(R+1)^2(1+\gamma)^2\alpha\tau(\alpha)}{(1-\gamma\lambda)^2}\cdot
\end{align*}
Third, using the first inequality in  \eqref{lem_xbar_bound_const:Ineq3} yields 
\begin{align*}
&\left|\Eset\left[(\btheta_{k}-\btheta_{k-\tau(\alpha)})^{T}(\Abf(X_{k}) -\Abf)(\btheta_{k}-\btheta_{k-\tau(\alpha)})\,|\,\Fcal_{k-\tau(\alpha)}\right]\right|\notag\\
&\quad \leq \Eset\left[\|\Abf(X_{k})-\Abf\|\|\btheta_{k}-\btheta_{k-\tau(\alpha)}\|^2\,|\,\Fcal_{k-\tau(\alpha)}\right]\notag\\
&\quad \leq \frac{2(1+\gamma)}{1-\gamma\lambda}\left(\frac{72(1+\gamma)^2\alpha^2\tau^2(\alpha)}{(1-\gamma\lambda)^2}\Eset\left[\|\btheta_{k}\|^2\,|\,\Fcal_{k-\tau(\alpha)}\right]  + \frac{72R^2\alpha^2\tau^2(\alpha)}{(1-\gamma\lambda)^2}\right)\notag\\
&\quad \leq \frac{48(1+\gamma)^2\alpha\tau(\alpha)}{(1-\gamma\lambda)^2} \Eset\left[\|\btheta_{k}\|^2 \,|\,\Fcal_{k-\tau(\alpha)}\right] + \frac{48R^2\alpha\tau(\alpha)}{(1-\gamma\lambda)^2},
\end{align*}
where in the last inequality we use Eq.\ \eqref{const_stepsize:alpha} to have
$(1+\gamma)\alpha\tau(\alpha)/(1-\gamma\lambda)\leq \log(2)\leq 1/3.$ Finally, using Eqs. \eqref{sec_analysis:Akbk_bound} and \eqref{sec_analysis:Ab_bound} we get Eq.\ \eqref{lem_opt_bound_const:Ineq4}, i.e.,
\begin{align*}
&\left|\Eset\left[(\btheta_{k}-\btheta_{k-\tau(\alpha)})^{T}(\bbar(X_{k}) - b)\,|\,\Fcal_{k-\tau(\alpha)}\right]\right|\notag\\
&\qquad \leq \frac{2R}{1-\gamma\lambda} E[\|\btheta_{k}-\btheta_{k-\tau(\alpha)}\|\,|\,\Fcal_{k-\tau(\alpha)}]\notag\\
&\qquad \stackrel{\eqref{lem_xbar_bound_const:Ineq2}}{\leq} \frac{2R}{1-\gamma\lambda}\left(\frac{6(1+\gamma)\alpha\tau(\alpha)}{1-\gamma\lambda}\|\btheta_{k}\|  + \frac{6R\alpha\tau(\alpha)}{1-\gamma\lambda}\right)\notag\\
&\qquad\leq \frac{12R(1+\gamma)\alpha\tau(\alpha)}{(1-\gamma\lambda)^2}\Eset[\|\theta_{k}\|\,|\,\Fcal_{k-\tau(\alpha)}] + \frac{12R^2\alpha\tau(\alpha)}{(1-\gamma\lambda)^2}\notag\\
&\qquad \leq \frac{6R(1+\gamma)\alpha\tau(\alpha)}{(1-\gamma\lambda)^2}\Eset[\|\theta_{k}\|^2\,|\,\Fcal_{k-\tau(\alpha)}] + \frac{12R(R+1)\alpha\tau(\alpha)}{(1-\gamma\lambda)^2},
\end{align*}
where the last inequality we use $2x\leq x^2 + 1$.
\end{proof}

Finally, using Lemmas \ref{lem_const:xbar_bound}--\ref{lem_const:opt_bound} we now ready to show Lemma \ref{lem_const:bias}.

\begin{proof}[Proof of Lemma \ref{lem_const:bias}]
First we consider 
\begin{align*}
&\Eset\left[(\btheta_{k}-\theta^*)^{T}(\Abf(X_{k})\btheta_{k} -\Abf\btheta_{k})\,|\,\Fcal_{k-\tau(\alpha)}\right]\\
& = \Eset\left[(\btheta_{k-\tau(\alpha)}-\theta^*)^{T}(\Abf(X_{k})-\Abf)\btheta_{k}\,|\,\Fcal_{k-\tau(\alpha)}\right] + \Eset\left[(\btheta_{k}-\btheta_{k-\tau(\alpha)})^{T}(\Abf(X_{k}) -\Abf)\btheta_{k}\,|\,\Fcal_{k-\tau(\alpha)}\right]\notag\\
&= \Eset\left[(\btheta_{k-\tau(\alpha)}-\theta^*)^{T}(\Abf(X_{k})-\Abf)\btheta_{k-\tau(\alpha)}\,|\,\Fcal_{k-\tau(\alpha)}\right]\notag\\ 
&\quad\qquad + \left[(\btheta_{k-\tau(\alpha)}-\theta^*)^{T}(\Abf(X_{k})-\Abf)(\btheta_{k}-\btheta_{k-\tau(\alpha)})\,|\,\Fcal_{k-\tau(\alpha)}\right]\notag\\ 
&\qquad + \Eset\left[(\btheta_{k}-\btheta_{k-\tau(\alpha)})^{T}(\Abf(X_{k}) -\Abf)\btheta_{k-\tau(\alpha)}\,|\,\Fcal_{k-\tau(\alpha)}\right]\notag\\ 
&\qquad + \Eset\left[(\btheta_{k}-\btheta_{k-\tau(\alpha)})^{T}(\Abf(X_{k}) -\Abf)(\btheta_{k}-\btheta_{k-\tau(\alpha)})\,|\,\Fcal_{k-\tau(\alpha)}\right], 
\end{align*}
which by using Eqs. \eqref{lem_mixing_bound:Ineq1a} and \eqref{lem_opt_bound_const:Ineq1}--\eqref{lem_opt_bound_const:Ineq3} we have
\begin{align}
&\Big|\Eset\left[(\btheta_{k}-\theta^*)^{T}(\Abf(X_{k})\btheta_{k} -\Abf\btheta_{k})\,|\,\Fcal_{k-\tau(\alpha)}\right]\Big|\notag\\
& \leq 27\alpha\Eset\left[\|\btheta_{k}\|^2\,|\,\Fcal_{k-\tau(\alpha)}\right] + (24R^2 + \|\theta^*\|^2)\alpha\notag\\
&\quad+ \frac{108(1+\gamma)^2\alpha\tau(\alpha)}{(1-\gamma\lambda)^2}\Eset\left[\|\btheta_{k}\|^2\,|\,\Fcal_{k-\tau(\alpha)}\right] +  \frac{100(1+\gamma)^2(R+\|\theta^*\|)^2\alpha\tau(\alpha)}{(1-\gamma\lambda)^2}\notag\\
&\quad + \frac{36(R+2)(1+\gamma)^2\alpha\tau(\alpha)}{(1-\gamma\lambda)^2}\Eset\left[\|\btheta_{k}\|^2\,|\,\Fcal_{k-\tau(\alpha)}\right] + \frac{32R(R+1)^2(1+\gamma)^2\alpha\tau(\alpha)}{(1-\gamma\lambda)^2}\notag\\
&\quad + \frac{48(1+\gamma)^2\alpha\tau(\alpha)}{(1-\gamma\lambda)^2} \Eset\left[\|\btheta_{k}\|^2 \,|\,\Fcal_{k-\tau(\alpha)}\right] + \frac{48R^2\alpha\tau(\alpha)}{(1-\gamma\lambda)^2}\notag\\
& \leq \left(27 + \frac{(228 + 36R)(1+\gamma)^2\tau(\alpha)}{(1-\gamma\lambda)^2}\right)\alpha \Eset\left[\|\btheta_{k}\|^2 \,|\,\Fcal_{k-\tau(\alpha)}\right]\notag\\
&\quad+ \left(24R^2+\|\theta^*\|^2 + \frac{\Big(48R^2+32R(R+1)^2+100(R+\|\theta^*\|)^2\Big)(1+\gamma)^2\tau(\alpha)}{(1-\gamma\lambda)^2}\right)\alpha.\label{thm_linear:Eq1a}
\end{align}
Similarly, we consider 
\begin{align*}
(\btheta_{k}-\theta^*)^{T}(\bbar(X_{k}) - b) = (\btheta_{k}-\btheta_{k-\tau(\alpha)})^{T}(\bbar(X_{k}) - b) + (\btheta_{k-\tau(\alpha)}-\theta^*)^{T}(\bbar(X_{k}) - b),    
\end{align*}
which by using Eqs.\ \eqref{lem_mixing_bound:Ineq1b} and \eqref{lem_opt_bound_const:Ineq4} yields
\begin{align}
&\Big|\Eset\left[(\btheta_{k}-\theta^*)^{T}(\bbar(X_{k}) - b)\,|\,\Fcal_{k-\tau(\alpha)}\right]\Big|\notag\\
&\leq 9\alpha\Eset[\|\btheta_{k}\|^2\,|\,\Fcal_{k-\tau(\alpha)}] + \left(8R^2 + 1 + \|\theta^*\|\right)\alpha + \frac{6R(1+\gamma)\alpha\tau(\alpha)}{(1-\gamma\lambda)^2}\Eset[\|\theta_{k}\|^2\,|\,\Fcal_{k-\tau(\alpha)}] + \frac{12R(R+1)\alpha\tau(\alpha)}{(1-\gamma\lambda)^2}\notag\\
&= \left(9 + \frac{6R(1+\gamma)\tau(\alpha)}{(1-\gamma\lambda)^2}\right)\alpha\Eset[\|\theta_{k}\|^2\,|\,\Fcal_{k-\tau(\alpha)}] + \left(\left(8R^2 + 1 + \|\theta^*\|\right) + \frac{12R(R+1)\tau(\alpha)}{(1-\gamma\lambda)^2} \right)\alpha.\label{thm_linear:Eq1b}
\end{align}
Using Eqs.\ \eqref{thm_linear:Eq1a} and \eqref{thm_linear:Eq1b} we obtain
\begin{align}
&\left|\Eset\left[(\btheta_{k}-\theta^*)^{T}(\Abf(X_{k})\btheta_{k} -\Abf\btheta_{k} + \bbar(X_{k}) - b)\,|\,\Fcal_{k-\tau(\alpha)}\right]\right|\notag\\
&\quad \leq \left(27 + \frac{(228 + 36R)(1+\gamma)^2\tau(\alpha)}{(1-\gamma\lambda)^2}\right)\alpha \Eset\left[\|\btheta_{k}\|^2 \,|\,\Fcal_{k-\tau(\alpha)}\right]\notag\\
&\quad\qquad + \left(24R^2+\|\theta^*\|^2 + \frac{\Big(48R^2+32R(R+1)^2+100(R+\|\theta^*\|)^2\Big)(1+\gamma)^2\tau(\alpha)}{(1-\gamma\lambda)^2}\right)\alpha\notag\\
&\quad\qquad + \left(9 + \frac{6R(1+\gamma)\tau(\alpha)}{(1-\gamma\lambda)^2}\right)\alpha\Eset[\|\theta_{k}\|^2\,|\,\Fcal_{k-\tau(\alpha)}] + \left(\left(8R^2 + 1 + \|\theta^*\|\right) + \frac{12R(R+1)\tau(\alpha)}{(1-\gamma\lambda)^2} \right)\alpha\notag\\
&\quad \leq \left(36 + \frac{(228 + 42R)(1+\gamma)^2\tau(\alpha)}{(1-\gamma\lambda)^2}\right)\alpha \Eset\left[\|\btheta_{k}\|^2 \,|\,\Fcal_{k-\tau(\alpha)}\right]\notag\\
&\quad\qquad + \left(32R^2+ 2\|\theta^*\|^2 + 1 +  \frac{\Big(48R^2+32(R+1)^3+100(R+\|\theta^*\|)^2\Big)(1+\gamma)^2\tau(\alpha)}{(1-\gamma\lambda)^2}\right)\alpha.\label{thm_linear:Eq1c}
\end{align}
\end{proof}

\section{Time-varying step sizes}\label{apx_tv}

Similarly, to show Lemma \ref{lem_tv:bias} we first consider the following sequence of lemmas.

\begin{lemma}\label{lem_tv:xbar_bound}
Suppose that Assumptions \ref{assump:doub_stoch}--\ref{assump:Markov} hold. Let the sequence $\{\theta_{k}^{v}\}$, for all $v\in\Vcal$, be generated by Algorithm \ref{alg:DTD} and let $\alpha_{k}$ satisfy the conditions in Lemma \ref{lem_tv:bias}. Denote by $\alpha_{k,\tau(\alpha_{k})}$
\begin{align}
\alpha_{k,\tau(\alpha_{k})} = \sum_{t = k-\tau(\alpha_{k})}^{k-1}\alpha_{t}.\label{notation:alpha_t_tau}    
\end{align}
Then for all $k\geq \Kcal^*$ we have
\begin{align}
&a)\qquad \|\btheta_{k}-\btheta_{k-\tau(\alpha_{k})}\|\leq \frac{2(1+\gamma)}{1-\gamma\lambda}\|\btheta_{k-\tau(\alpha_{k})}\|\alpha_{k,\tau(\alpha_{k})} + \frac{2R}{1-\gamma\lambda}\alpha_{k,\tau(\alpha_{k})}.\label{lem_tv_xbar_bound:Ineq1}\\
&b)\qquad \|\btheta_{k}-\btheta_{k-\tau(\alpha_{k})}\|\leq \frac{6(1+\gamma)}{1-\gamma\lambda}\|\theta_{k}\|\alpha_{k,\tau(\alpha_{k})} + \frac{6R}{1-\gamma\lambda}\alpha_{k,\tau(\alpha_{k})}.\label{lem_tv_xbar_bound:Ineq2}\\
&c)\qquad \|\btheta_{k}-\btheta_{k-\tau(\alpha_{k})}\|^2 \leq \frac{72(1+\gamma)^2}{(1-\gamma\lambda)^2}\alpha_{k,\tau(\alpha_{k})}^2\|\btheta_{k}\|^2 + \frac{72R^2}{(1-\gamma\lambda)^2}\alpha_{k,\tau(\alpha_{k})}^2 \leq 8\|\btheta_{k}\|^2 + 8R^2.\label{lem_tv_xbar_bound:Ineq3}
\end{align}
\end{lemma}

\begin{proof}
Let $\tau(\alpha_{k})$ be the mixing time associated with $\alpha_k$ defined in \eqref{def_mixing:tau}. First, using Eq.\ \eqref{sec_analysis:theta_bar}  we have
\begin{align*}
\btheta_{k+1} 
&= \btheta_{k} + \alpha_{k}\Abf(X_{k})\btheta_{k} + \alpha_{k}\bbar(X_{k}),
\end{align*}
which when taking the $2$-norm on both sides and using Eq.\ \eqref{sec_analysis:Akbk_bound} yields
\begin{align*}
\|\btheta_{k+1}\| &\leq \|\btheta_{k}\| + \frac{(1+\gamma)\alpha_{k}}{1-\gamma\lambda}\|\btheta_{k}\| + \frac{R\alpha_{k}}{1-\gamma\lambda} = \left(1 + \frac{(1+\gamma)\alpha_{k}}{1-\gamma\lambda}\right)\|\btheta_{k}\| + \frac{R\alpha_{k}}{1-\gamma\lambda}\cdot
\end{align*}
Using the relation $1+x \leq \exp(x)$ for all $x\geq0$ we have from the preceding relation for some $t\in[k-\tau(\alpha_k),\;k-1]$
\begin{align*}
\|\btheta_{t}\| &\leq \prod_{u = k-\tau(\alpha_k)}^{t-1}\left(1 + \frac{(1+\gamma)\alpha_{u}}{1-\gamma\lambda}\right)\|\btheta_{k-\tau(\alpha_k)}\| + \frac{R}{1-\gamma\lambda}\sum_{u=k-\tau(\alpha_k)}^{t-1}\alpha_{u}\prod_{\ell = u+1}^{t-1}\left(1 + \frac{(1+\gamma)\alpha_{\ell}}{1-\gamma\lambda}\right)\notag\\
&\leq \|\btheta_{k-\tau(\alpha_k)}\|\exp\left\{\sum_{u=k-\tau(\alpha_k)}^{t-1} \frac{(1+\gamma)\alpha_{u}}{1-\gamma\lambda}\right\} + \frac{R}{1-\gamma\lambda} \sum_{u=k-\tau(\alpha_k)}^{t-1}\alpha_{u}\exp\left\{\sum_{\ell=u+1}^{t-1}\frac{(1+\gamma)\alpha_{\ell}}{1-\gamma\lambda}\right\}\notag\\
&\leq \|\btheta_{k-\tau(\alpha_k)}\|\exp\left\{\frac{(1+\gamma)\alpha_{k-\tau(\alpha_k)}\tau(\alpha_k)}{1-\gamma\lambda}\right\}\notag\\
&\quad\qquad + \frac{R}{1-\gamma\lambda} \sum_{u=k-\tau(\alpha_k)}^{t-1}\alpha_{u}\exp\left\{\frac{(1+\gamma)\alpha_{k-\tau(\alpha_k)}\tau(\alpha_k)}{1-\gamma\lambda}\right\}.
\end{align*}
Recall from \eqref{thm_sublinear:stepsize} that $\Kcal^*$ is a positive integer such that 
\begin{align*}
\alpha_{k-\tau(\alpha_k)}\tau(\alpha_k) \leq \frac{(1-\gamma\lambda)\log(2)}{1+\gamma},\qquad k\geq \Kcal^*, 
\end{align*}
which by using the equation above we have for all $t\in[k-\tau(\alpha_k),\; k-1]$ with $k\geq \Kcal^*$
\begin{align*}
\|\btheta_{t}\| &\leq 2\|\btheta_{k-\tau(\alpha_k)}\| + \frac{2R}{1-\gamma\lambda} \sum_{u=k-\tau(\alpha_k)}^{t-1}\alpha_{u}.
\end{align*}
Thus, using the preceding relation yields Eq.\ \eqref{lem_tv_xbar_bound:Ineq1}, i.e., for all $k\geq \Kcal^*$
\begin{align*}
&\|\btheta_{k}-\btheta_{k-\tau(\alpha_k)}\|\leq \sum_{t = k-\tau{\alpha_k}}^{k}\|\btheta_{t+1} - \btheta_{t}\| \leq \frac{1+\gamma}{1-\gamma\lambda}\sum_{t = k-\tau{\alpha_k}}^{k}\alpha_{t}\|\btheta_{t}\| + \frac{R}{1-\gamma\lambda}\sum_{t = k-\tau(\alpha_k)}^{k}\alpha_{t}\notag\\
&\leq \frac{1+\gamma}{1-\gamma\lambda}\sum_{t = k-\tau{\alpha_k}}^{k}\alpha_{t}\left(2\|\btheta_{k-\tau(\alpha_k)}\| + \frac{2R}{1-\gamma\lambda} \sum_{u=k-\tau(\alpha_k)}^{t-1}\alpha_{u}\right) + \frac{R}{1-\gamma\lambda}\sum_{t = k-\tau(\alpha_k)}^{k}\alpha_{t}\notag\\
&\leq \frac{2(1+\gamma)}{1-\gamma\lambda}\|\btheta_{k-\tau(\alpha_k)}\|\sum_{t = k-\tau{\alpha_k}}^{k}\alpha_{t} + \frac{R}{1-\gamma\lambda}\left(\frac{2(1+\gamma)\alpha_{k-\tau(\alpha_k)}\tau(\alpha_k)}{1-\gamma\lambda}+1\right)\sum_{t = k-\tau{\alpha_k}}^{k}\alpha_{t}\notag\\
&\leq \frac{2(1+\gamma)}{1-\gamma\lambda}\|\btheta_{k-\tau(\alpha_k)}\|\sum_{t = k-\tau{\alpha_k}}^{k}\alpha_{t} + \frac{2R}{1-\gamma\lambda}\sum_{t = k-\tau{\alpha_k}}^{k}\alpha_{t},
\end{align*}
where the last inequality we use 
\begin{align}
\frac{(1+\gamma)\alpha_{k-\tau(\alpha_k)}\tau(\alpha_k)}{1-\gamma\lambda} \leq \log(2)\leq \frac{1}{3},\qquad k\geq \Kcal^*.\label{lem_tv_xbar_bound:Eq1a} 
\end{align}
Next, by triangle inequality we have
\begin{align*}
\|\btheta_{k}-\btheta_{k-\tau(\alpha_k)}\| &\leq  \frac{2(1+\gamma)}{1-\gamma\lambda}\|\btheta_{k}-\btheta_{k-\tau(\alpha_k)}\|\sum_{t = k-\tau{\alpha_k}}^{k}\alpha_{t} + \frac{2(1+\gamma)}{1-\gamma\lambda}\|\btheta_{k}\|\sum_{t = k-\tau{\alpha_k}}^{k}\alpha_{t} + \frac{2R}{1-\gamma\lambda}\sum_{t = k-\tau{\alpha_k}}^{k}\alpha_{t},   
\end{align*}
which by using Eq.\ \eqref{lem_tv_xbar_bound:Eq1a} gives Eq.\ \eqref{lem_tv_xbar_bound:Ineq2}, i.e., for all $k\geq \Kcal^*$
\begin{align*}
\|\btheta_{k}-\btheta_{k-\tau(\alpha_k)}\| \leq  \frac{6(1+\gamma)}{1-\gamma\lambda}\|\btheta_{k}\|\sum_{t = k-\tau{\alpha_k}}^{k}\alpha_{t} + \frac{6R}{1-\gamma\lambda}\sum_{t = k-\tau{\alpha_k}}^{k}\alpha_{t}.
\end{align*}
Finally, using the inequality $(a+b)^2 \leq 2a^2 + 2b^2$ and Eq.\ \eqref{lem_tv_xbar_bound:Eq1a} we obtain Eq.\ \eqref{lem_tv_xbar_bound:Ineq3}, i.e., for all $k\geq \Kcal^*$
\begin{align*}
\|\btheta_{k}-\btheta_{k-\tau(\alpha_k)}\|^2 &\leq \frac{72(1+\gamma)^2}{(1-\gamma\lambda)^2}\|\btheta_{k}\|^2\left(\sum_{t = k-\tau{\alpha_k}}^{k}\alpha_{t}\right)^2 + \frac{72R^2}{(1-\gamma\lambda)^2}\left(\sum_{t = k-\tau{\alpha_k}}^{k}\alpha_{t}\right)^2 \leq 8\|\btheta_{k}\|^2 + 8R^2.
\end{align*}
\end{proof}

\begin{lemma}\label{lem_mixing_tv:bound}
Suppose that Assumptions \ref{assump:doub_stoch}--\ref{assump:Markov} hold. Let the sequence $\{\theta_{k}^{v}\}$, for all $v\in\Vcal$, be generated by Algorithm \ref{alg:DTD}. Let $\alpha_{k}$ satisfy all the conditions in Lemma \ref{lem_tv:xbar_bound}. Then for all $k\geq \Kcal^*$ we have
\begin{align}
&\left|\Eset\left[(\btheta_{k-\tau(\alpha_k)}-\theta^*)^{T}(\Abf(X_{k})-\Abf)\btheta_{k-\tau(\alpha_k)}\,|\,\Fcal_{k-\tau(\alpha_k)}\right]\right| \leq 27\alpha_{k}\Eset\left[\|\btheta_{k}\|^2\,|\,\Fcal_{k-\tau(\alpha_{k})}\right] + (24R^2 + \|\theta^*\|^2)\alpha_{k}.\label{lem_mixing_tv_bound:Ineq1}\\
&\left|\Eset\left[(\btheta_{k-\tau(\alpha_k)}-\theta^*)^{T}(\bbar(X_{k}) - b)\,|\,\Fcal_{k-\tau(\alpha_k)}\right]\right|
\leq 9\alpha_{k}\Eset[\|\btheta_{k}\|^2\,|\,\Fcal_{k-\tau(\alpha_{k})}] + \left(8R^2 + 1 + \|\theta^*\|\right)\alpha_{k}.\label{lem_mixing_tv_bound:Ineq2}
\end{align}
\end{lemma}

\begin{proof}
Recall that $\Fcal_{k}$ is the filtration containing all the history up to time $k$. Let $\tau(\alpha_k)$ be the mixing time associated with step size $\alpha_k$ defined in \eqref{def_mixing:tau}. Then we first have $k\geq\tau(\alpha_{k})$
\begin{align*}
&\left|\Eset\left[(\btheta_{k-\tau(\alpha_{k})}-\theta^*)^{T}(\Abf(X_{k})-\Abf)\btheta_{k-\tau(\alpha_{k})}\,|\,\Fcal_{k-\tau(\alpha_{k})}\right]\right|\notag\\ 
&\quad =  \left|(\btheta_{k-\tau(\alpha_{k})}-\theta^*)^{T}\Eset\left[\Abf(X_{k})-\Abf\,|\,\Fcal_{k-\tau(\alpha_{k})}\right]\btheta_{k-\tau(\alpha_{k})}\right|\notag\\ 
&\quad \stackrel{\eqref{def_mixing:tau}}{\leq} \alpha_{k}\Eset\left[\|\btheta_{k-\tau(\alpha_{k})}-\theta^*\|\,\|\btheta_{k-\tau(\alpha_{k})}\|\,|\,\Fcal_{k-\tau(\alpha_{k})}\right] \leq \frac{\alpha_{k}}{2}\Eset\left[\|\btheta_{k-\tau(\alpha_{k})}\|^2 + \|\btheta_{k-\tau(\alpha_{k})}-\theta^*\|^2\,|\,\Fcal_{k-\tau(\alpha_{k})}\right]\notag\\ 
&\quad \leq \frac{3\alpha_{k}}{2}\Eset\left[\|\btheta_{k-\tau(\alpha_{k})}\|^2\,|\,\Fcal_{k-\tau(\alpha_{k})}\right] + \|\theta^*\|^2\alpha_{k}\notag\\
&\quad \leq 3\alpha_{k}\Eset\left[\|\btheta_{k}-\btheta_{k-\tau(\alpha_{k})}\|^2\,|\,\Fcal_{k-\tau(\alpha_{k})}\right] + 3\alpha_{k}\Eset\left[\|\btheta_{k}\|^2\,|\,\Fcal_{k-\tau(\alpha_{k})}\right] +  \|\theta^*\|^2\alpha_{k},
\end{align*}
which by using Eq.\ \eqref{lem_tv_xbar_bound:Ineq3} yields Eq.\   \eqref{lem_mixing_tv_bound:Ineq1}, i.e.,
\begin{align*}
&\left|\Eset\left[(\btheta_{k-\tau(\alpha_{k})}-\theta^*)^{T}(\Abf(X_{k})-\Abf)\btheta_{k-\tau(\alpha_{k})}\,|\,\Fcal_{k-\tau(\alpha_{k})}\right]\right|\notag\\
&\quad \leq 3\alpha_{k}\left(8\|\btheta_{k}\|^2  + 8R^2\right) + 3\alpha_{k}\Eset\left[\|\btheta_{k}\|^2\,|\,\Fcal_{k-\tau(\alpha_{k})}\right] +  \|\theta^*\|^2\alpha_{k}\notag\\
&\quad \leq 27\alpha_{k}\Eset\left[\|\btheta_{k}\|^2\,|\,\Fcal_{k-\tau(\alpha_{k})}\right] + (24R^2 + \|\theta^*\|^2)\alpha_{k}.
\end{align*}
Next, we use the definition of mixing time in Eq.\ \eqref{def_mixing:tau} again to have
\begin{align*}
&\left|\Eset\left[(\btheta_{k-\tau(\alpha_{k})}-\theta^*)^{T}(\bbar(X_{k}) - b)\,|\,\Fcal_{k-\tau(\alpha_{k})}\right]\right| = \left|(\btheta_{k-\tau(\alpha_{k})}-\theta^*)^{T}\Eset\left[\bbar(X_{k}) - b\,|\,\Fcal_{k-\tau(\alpha_{k})}\right]\right|\notag\\
&\qquad \leq \alpha_{k}\Eset[\|\btheta_{k-\tau(\alpha_{k})}-\theta^*\|\,|\,\Fcal_{k-\tau(\alpha_{k})}] \leq \alpha_{k}\Eset[\|\btheta_{k-\tau(\alpha_{k})}\|\,|\,\Fcal_{k-\tau(\alpha_{k})}] + \|\theta^*\|\alpha_{k}\notag\\
&\qquad \leq \frac{\alpha_{k}}{2}\Eset[\|\btheta_{k-\tau(\alpha_{k})}\|^2\,|\,\Fcal_{k-\tau(\alpha_{k})}] + \left(\frac{1}{2} + \|\theta^*\|\right)\alpha_{k}\notag\\
&\qquad \leq \alpha_{k}\Eset[\|\btheta_{k} - \btheta_{k-\tau(\alpha_{k})}\|^2\,|\,\Fcal_{k-\tau(\alpha_{k})}] + \alpha_{k}\Eset[\|\btheta_{k}\|^2\,|\,\Fcal_{k-\tau(\alpha_{k})}] + \left(\frac{1}{2} + \|\theta^*\|\right)\alpha_{k},
\end{align*}
which by using Eq.\ \eqref{lem_tv_xbar_bound:Ineq3} yields Eq.\ \eqref{lem_mixing_tv_bound:Ineq2}, i.e.,
\begin{align*}
&\left|\Eset\left[(\btheta_{k-\tau(\alpha_{k})}-\theta^*)^{T}(\bbar(X_{k}) - b)\,|\,\Fcal_{k-\tau(\alpha_{k})}\right]\right|\notag\\ 
&\qquad \leq \alpha_{k}\left(8\Eset[\|\btheta_{k}\|^2\,|\,\Fcal_{k-\tau(\alpha_{k})}] + 8R^2\right) +  \alpha_{k}\Eset[\|\btheta_{k}\|^2\,|\,\Fcal_{k-\tau(\alpha_{k})}] + \left(\frac{1}{2} + \|\theta^*\|\right)\alpha_{k}\notag\\
&\qquad \leq 9\alpha_{k}\Eset[\|\btheta_{k}\|^2\,|\,\Fcal_{k-\tau(\alpha_{k})}] + \left(8R^2 + 1 + \|\theta^*\|\right)\alpha_{k}.
\end{align*}
\end{proof}

\begin{lemma}\label{lem_tv:opt_bound}
Suppose that Assumptions \ref{assump:doub_stoch}--\ref{assump:Markov} hold. Let the sequence $\{\theta_{k}^{v}\}$, for all $v\in\Vcal$, be generated by Algorithm \ref{alg:DTD}. Let $\alpha_{k}$ satisfy all the conditions in Lemma \ref{lem_tv:xbar_bound}. Then for all $k\geq \Kcal^*$ we have
\begin{align}
&\left |\Eset\left[(\btheta_{k-\tau(\alpha_k)}-\theta^*)^{T}(\Abf(X_{k})-\Abf)(\btheta_{k}-\btheta_{k-\tau(\alpha)})\,|\,\Fcal_{k-\tau(\alpha_k)}\right]\right|\notag\\
&\qquad\qquad \leq \frac{108(1+\gamma)^2\alpha_{k,\tau(\alpha_k)}}{(1-\gamma\lambda)^2}\Eset\left[\|\btheta_{k}\|^2\,|\,\Fcal_{k-\tau(\alpha_k)}\right] +  \frac{100(1+\gamma)^2(R+\|\theta^*\|)^2\alpha_{k,\tau(\alpha_k)}}{(1-\gamma\lambda)^2}\cdot\label{lem_tv_opt_bound:Ineq1}\\
&\left|\Eset\left[(\btheta_{k}-\btheta_{k-\tau(\alpha_k)})^{T}(\Abf(X_{k}) -\Abf)\btheta_{k-\tau(\alpha_k)}\,|\,\Fcal_{k-\tau(\alpha_k)}\right]\right|\notag\\
&\qquad\qquad \leq \frac{36(R+2)(1+\gamma)^2\alpha_{k,\tau(\alpha_k)}}{(1-\gamma\lambda)^2}\Eset\left[\|\btheta_{k}\|^2\,|\,\Fcal_{k-\tau(\alpha_{k})}\right] + \frac{32R(R+1)^2(1+\gamma)^2\alpha_{k,\tau(\alpha_k)}}{(1-\gamma\lambda)^2}\cdot\label{lem_tv_opt_bound:Ineq2}\\
&\left|\Eset\left[(\btheta_{k}-\btheta_{k-\tau(\alpha_{k})})^{T}(\Abf(X_{k}) -\Abf)(\btheta_{k}-\btheta_{k-\tau(\alpha_{k})})\,|\,\Fcal_{k-\tau(\alpha_{k})}\right]\right|\notag\\
&\qquad\qquad \leq \frac{48(1+\gamma)^2\alpha_{k,\tau(\alpha_k)}}{(1-\gamma\lambda)^2} \Eset\left[\|\btheta_{k}\|^2 \,|\,\Fcal_{k-\tau(\alpha_{k})}\right] + \frac{48R^2\alpha_{k,\tau(\alpha_k)}}{(1-\gamma\lambda)^2}\cdot\label{lem_tv_opt_bound:Ineq3}\\
&\left|\Eset\left[(\btheta_{k}-\btheta_{k-\tau(\alpha_{k})})^{T}(\bbar(X_{k}) - b)\,|\,\Fcal_{k-\tau(\alpha_{k})}\right]\right|\notag\\
&\qquad\qquad \leq \frac{6R(1+\gamma)\alpha_{k,\tau(\alpha_k)}}{(1-\gamma\lambda)^2}\Eset[\|\theta_{k}\|^2\,|\,\Fcal_{k-\tau(\alpha_{k})}] + \frac{12R(R+1)\alpha_{k,\tau(\alpha_k)}}{(1-\gamma\lambda)^2}\cdot\label{lem_tv_opt_bound:Ineq4}
\end{align}
\end{lemma}

\begin{proof}
First, we show Eq.\ \eqref{lem_tv_opt_bound:Ineq1}. Using Eqs.\ \eqref{sec_analysis:Akbk_bound} and \eqref{sec_analysis:Ab_bound} we have for all $k\geq \tau(\alpha_k)$
\begin{align*}
&\left |\Eset\left[(\btheta_{k-\tau(\alpha_{k})}-\theta^*)^{T}(\Abf(X_{k})-\Abf)(\btheta_{k}-\btheta_{k-\tau(\alpha_{k})})\,|\,\Fcal_{k-\tau(\alpha_{k})}\right]\right|\notag\\
&\quad \leq (\|\Abf(X_{k})\|+\|\Abf\|)\Eset\left[\|\btheta_{k-\tau(\alpha_{k})}-\theta^*\|\|\btheta_{k}-\btheta_{k-\tau(\alpha_{k})}\|\,|\,\Fcal_{k-\tau(\alpha_{k})}\right]\notag\\
&\quad \leq \frac{2(1+\gamma)}{1-\gamma\lambda}\Eset\left[\|\btheta_{k-\tau(\alpha_{k})}\|\|\btheta_{k}-\btheta_{k-\tau(\alpha_{k})}\|\,|\,\Fcal_{k-\tau(\alpha_{k})}\right] + \frac{2\|\theta^*\|(1+\gamma)}{1-\gamma\lambda}\Eset\left[\|\btheta_{k}-\btheta_{k-\tau(\alpha_{k})}\|\,|\,\Fcal_{k-\tau(\alpha_{k})}\right],
\end{align*}
which by using Eq.\ \eqref{lem_tv_xbar_bound:Ineq1} to upper bound the term $\|\btheta_{k}-\btheta_{k-\tau(\alpha_{k})}\|$ yields
\begin{align*}
&\left|\Eset\left[(\btheta_{k-\tau(\alpha_{k})}-\theta^*)^{T}(\Abf(X_{k})-\Abf)(\btheta_{k}-\btheta_{k-\tau(\alpha_{k})})\,|\,\Fcal_{k-\tau(\alpha_{k})}\right]\right|\notag\\
&\quad \leq \frac{4(1+\gamma)^2\alpha_{k,\tau(\alpha_k)}}{(1-\gamma\lambda)^2}\Eset\left[\|\btheta_{k-\tau(\alpha_{k})}\|^2\,|\,\Fcal_{k-\tau(\alpha_{k})}\right]\notag\\ 
&\qquad\quad+ \frac{4(1+\gamma)R\alpha_{k,\tau(\alpha_k)}}{(1-\gamma\lambda)^2}\Eset\left[\|\btheta_{k-\tau(\alpha_{k})}\|\,|\,\Fcal_{k-\tau(\alpha_{k})}\right]\notag\\ 
&\qquad\quad + \frac{2\|\theta^*\|(1+\gamma)}{1-\gamma\lambda}\left(\frac{2(1+\gamma)\alpha_{k,\tau(\alpha_k)}}{1-\gamma\lambda}\Eset\left[\|\btheta_{k-\tau(\alpha_{k})}\|\,|\,\Fcal_{k-\tau(\alpha_{k})}\right] + \frac{2R\alpha_{k,\tau(\alpha_k)}}{1-\gamma\lambda}\right)\notag\\
&\quad= \frac{4(1+\gamma)^2\alpha_{k,\tau(\alpha_k)}}{(1-\gamma\lambda)^2}\Eset\left[\|\btheta_{k-\tau(\alpha_{k})}\|^2\,|\,\Fcal_{k-\tau(\alpha_{k})}\right]\notag\\ 
&\quad \qquad  + \frac{4(1+\gamma)^2(R+\|\theta^*\|)\alpha_{k,\tau(\alpha_k)}}{(1-\gamma\lambda)^2}\Eset\left[\|\btheta_{k-\tau(\alpha_{k})}\|\,|\,\Fcal_{k-\tau(\alpha_{k})}\right] + \frac{4\|\theta^*\|R(1+\gamma)\alpha_{k,\tau(\alpha_k)}}{(1-\gamma\lambda)^2}\cdot
\end{align*}
Using the relation $2xy\leq x^2 + y^2$ to bound the second and third terms on the right-hand side of the preceding relation yields
\begin{align*}
&\left|\Eset\left[(\btheta_{k-\tau(\alpha_{k})}-\theta^*)^{T}(\Abf(X_{k})-\Abf)(\btheta_{k}-\btheta_{k-\tau(\alpha_{k})})\,|\,\Fcal_{k-\tau(\alpha_{k})}\right]\right|\notag\\
&\quad \leq \frac{6(1+\gamma)^2\alpha_{k,\tau(\alpha_k)}}{(1-\gamma\lambda)^2}\Eset\left[\|\btheta_{k-\tau(\alpha_{k})}\|^2\,|\,\Fcal_{k-\tau(\alpha_{k})}\right] +  \frac{2(1+\gamma)^2(R+\|\theta^*\|)^2\alpha_{k,\tau(\alpha_k)}}{(1-\gamma\lambda)^2}\notag\\
&\quad\qquad + \frac{2(1+\gamma)(R+\|\theta^*\|)^2\alpha_{k,\tau(\alpha_k)}}{(1-\gamma\lambda)^2}\notag\\
&\quad \leq \frac{6(1+\gamma)^2\alpha_{k,\tau(\alpha_k)}}{(1-\gamma\lambda)^2}\Eset\left[\|\btheta_{k-\tau(\alpha_{k})}\|^2\,|\,\Fcal_{k-\tau(\alpha_{k})}\right] +  \frac{4(1+\gamma)^2(R+\|\theta^*\|)^2\alpha_{k,\tau(\alpha_k)}}{(1-\gamma\lambda)^2}\notag\\
&\quad \leq \frac{12(1+\gamma)^2\alpha_{k,\tau(\alpha_k)}}{(1-\gamma\lambda)^2}\left(\Eset\left[\|\btheta_{k}-\btheta_{k-\tau(\alpha_{k})}\|^2\,|\,\Fcal_{k-\tau(\alpha_{k})}\right] + \Eset\left[\|\btheta_{k}\|^2\,|\,\Fcal_{k-\tau(\alpha_{k})}\right]\right)\notag\\ 
&\quad\qquad +  \frac{4(1+\gamma)^2(R+\|\theta^*\|)^2\alpha_{k,\tau(\alpha_k)}}{(1-\gamma\lambda)^2} ,
\end{align*}
which by using Eq.\ \eqref{lem_tv_xbar_bound:Ineq3} to the first term on the right-hand side we obtain Eq.\ \eqref{lem_tv_opt_bound:Ineq1}, i.e.,
\begin{align*}
&\left |\Eset\left[(\btheta_{k-\tau(\alpha_{k})}-\theta^*)^{T}(\Abf(X_{k})-\Abf)(\btheta_{k}-\btheta_{k-\tau(\alpha_{k})})\,|\,\Fcal_{k-\tau(\alpha_{k})}\right]\right|\notag\\
&\quad \leq \frac{12(1+\gamma)^2\alpha_{k,\tau(\alpha_k)}}{(1-\gamma\lambda)^2}\left(8\Eset\left[\|\btheta_{k}\|^2\,|\,\Fcal_{k-\tau(\alpha_{k})}\right] + 8R^2 \right)\notag\\
&\quad\qquad + \frac{12(1+\gamma)^2\alpha_{k,\tau(\alpha_k)}}{(1-\gamma\lambda)^2}\Eset\left[\|\btheta_{k}\|^2\,|\,\Fcal_{k-\tau(\alpha_{k})}\right] +  \frac{4(1+\gamma)^2(R+\|\theta^*\|)^2\alpha_{k,\tau(\alpha_k)}}{(1-\gamma\lambda)^2} \notag\\
&\quad \leq \frac{108(1+\gamma)^2\alpha_{k,\tau(\alpha_k)}}{(1-\gamma\lambda)^2}\Eset\left[\|\btheta_{k}\|^2\,|\,\Fcal_{k-\tau(\alpha_{k})}\right] +  \frac{100(1+\gamma)^2(R+\|\theta^*\|)^2\alpha_{k,\tau(\alpha_k)}}{(1-\gamma\lambda)^2}\cdot
\end{align*}
Next, using Eqs.\ \eqref{sec_analysis:Akbk_bound} and  \eqref{sec_analysis:Ab_bound} again we have
\begin{align*}
&\left|\Eset\left[(\btheta_{k}-\btheta_{k-\tau(\alpha_{k})})^{T}(\Abf(X_{k}) -\Abf)\btheta_{k-\tau(\alpha_{k})}\,|\,\Fcal_{k-\tau(\alpha_{k})}\right]\right|\notag\\
&\quad \leq  \frac{2(1+\gamma)}{1-\gamma\lambda}\Eset\left[\|\btheta_{k-\tau(\alpha_{k})}\|\|\btheta_{k}-\btheta_{k-\tau(\alpha_{k})}\|\,|\,\Fcal_{k-\tau(\alpha_{k})}\right]
\end{align*}
which by using Eq.\ \eqref{lem_tv_xbar_bound:Ineq1} and $2x\leq x^2 + 1$ yields
\begin{align*}
&\left|\Eset\left[(\btheta_{k}-\btheta_{k-\tau(\alpha_{k})})^{T}(\Abf(X_{k}) -\Abf)\btheta_{k-\tau(\alpha_{k})}\,|\,\Fcal_{k-\tau(\alpha_{k})}\right]\right|\notag\\
&\quad\leq \frac{2(1+\gamma)}{1-\gamma\lambda}\Eset\left[\frac{2(1+\gamma)\alpha_{k,\tau(\alpha_k)}}{1-\gamma\lambda}\|\btheta_{k-\tau(\alpha_{k})}\|^2  + \frac{2R\alpha_{k,\tau(\alpha_k)}}{1-\gamma\lambda}\|\btheta_{k-\tau(\alpha_{k})}\|\,\Bigg|\,\Fcal_{k-\tau(\alpha_{k})}\right]\\
&\quad\leq \frac{2(1+\gamma)}{1-\gamma\lambda}\Eset\left[\frac{(R+2)(1+\gamma)\alpha_{k,\tau(\alpha_k)}}{1-\gamma\lambda}\|\btheta_{k-\tau(\alpha_{k})}\|^2  + \frac{R\alpha_{k,\tau(\alpha_k)}}{1-\gamma\lambda}\,\Bigg|\,\Fcal_{k-\tau(\alpha_{k})}\right]\\
&\quad\leq  \frac{4(R+2)(1+\gamma)^2\alpha_{k,\tau(\alpha_k)}}{(1-\gamma\lambda)^2}\Eset\left[\|\btheta_{k}-\btheta_{k-\tau(\alpha_{k})}\|^2  + \|\btheta_{k}\|^2\,\big|\,\Fcal_{k-\tau(\alpha_{k})}\right] + \frac{2R(1+\gamma)\alpha_{k,\tau(\alpha_k)}}{(1-\gamma\lambda)^2}.
\end{align*}
Using Eq.\ \eqref{lem_tv_xbar_bound:Ineq3} to bound the preceding relation gives Eq.\ \eqref{lem_tv_opt_bound:Ineq2}, i.e.,
\begin{align*}
&\left|\Eset\left[(\btheta_{k}-\btheta_{k-\tau(\alpha_{k})})^{T}(\Abf(X_{k}) -\Abf)\btheta_{k-\tau(\alpha_{k})}\,|\,\Fcal_{k-\tau(\alpha_{k})}\right]\right|\notag\\
&\quad\leq \frac{4(R+2)(1+\gamma)^2\alpha_{k,\tau(\alpha_k)}}{(1-\gamma\lambda)^2}\Eset\left[9\|\btheta_{k}\|^2 + 8R^2\,|\,\Fcal_{k-\tau(\alpha_{k})}\right] + \frac{2R(1+\gamma)\alpha_{k,\tau(\alpha_k)}}{(1-\gamma\lambda)^2}\\
&\quad \leq \frac{36(R+2)(1+\gamma)^2\alpha_{k,\tau(\alpha_k)}}{(1-\gamma\lambda)^2}\Eset\left[\|\btheta_{k}\|^2\,|\,\Fcal_{k-\tau(\alpha_{k})}\right] + \frac{32R(R+1)^2(1+\gamma)^2\alpha_{k,\tau(\alpha_k)}}{(1-\gamma\lambda)^2}\cdot
\end{align*}
Third, using the first inequality in  \eqref{lem_tv_xbar_bound:Ineq3} yields Eq.\ \eqref{lem_tv_opt_bound:Ineq3}, i.e.,
\begin{align*}
&\left|\Eset\left[(\btheta_{k}-\btheta_{k-\tau(\alpha_{k})})^{T}(\Abf(X_{k}) -\Abf)(\btheta_{k}-\btheta_{k-\tau(\alpha_{k})})\,|\,\Fcal_{k-\tau(\alpha_{k})}\right]\right|\notag\\
&\quad \leq \Eset\left[\|\Abf(X_{k})-\Abf\|\|\btheta_{k}-\btheta_{k-\tau(\alpha_{k})}\|^2\,|\,\Fcal_{k-\tau(\alpha_{k})}\right]\notag\\
&\quad \leq \frac{2(1+\gamma)}{1-\gamma\lambda}\left(\frac{72(1+\gamma)^2\alpha_{k}^2\tau^2(\alpha_{k})}{(1-\gamma\lambda)^2}\Eset\left[\|\btheta_{k}\|^2\,|\,\Fcal_{k-\tau(\alpha_{k})}\right]  + \frac{72R^2\alpha_{k}^2\tau^2(\alpha_{k})}{(1-\gamma\lambda)^2}\right)\notag\\
&\quad \leq \frac{48(1+\gamma)^2\alpha_{k,\tau(\alpha_k)}}{(1-\gamma\lambda)^2} \Eset\left[\|\btheta_{k}\|^2 \,|\,\Fcal_{k-\tau(\alpha_{k})}\right] + \frac{48R^2\alpha_{k,\tau(\alpha_k)}}{(1-\gamma\lambda)^2},
\end{align*}
where in the last inequality we use Eq.\ \eqref{thm_sublinear:stepsize} to have $
(1+\gamma)\alpha_{k,\tau(\alpha_k)}/(1-\gamma\lambda)\leq \log(2)\leq \frac{1}{3},\; \forall k \geq\Kcal^*.$ Finally, using Eqs. \eqref{sec_analysis:Akbk_bound} and \eqref{sec_analysis:Ab_bound} we get Eq.\ \eqref{lem_tv_opt_bound:Ineq4}, i.e.,
\begin{align*}
&\left|\Eset\left[(\btheta_{k}-\btheta_{k-\tau(\alpha_{k})})^{T}(\bbar(X_{k}) - b)\,|\,\Fcal_{k-\tau(\alpha_{k})}\right]\right| \leq \frac{2R}{1-\gamma\lambda} E[\|\btheta_{k}-\btheta_{k-\tau(\alpha_{k})}\|\,|\,\Fcal_{k-\tau(\alpha_{k})}]\notag\\
& \stackrel{\eqref{lem_tv_xbar_bound:Ineq2}}{\leq} \frac{2R}{1-\gamma\lambda}\left(\frac{6(1+\gamma)\alpha_{k,\tau(\alpha_k)}}{1-\gamma\lambda}\|\btheta_{k}\|  + \frac{6R\alpha_{k,\tau(\alpha_k)}}{1-\gamma\lambda}\right)\notag\\
&\leq \frac{12R(1+\gamma)\alpha_{k,\tau(\alpha_k)}}{(1-\gamma\lambda)^2}\Eset[\|\theta_{k}\|\,|\,\Fcal_{k-\tau(\alpha_{k})}] + \frac{12R^2\alpha_{k,\tau(\alpha_k)}}{(1-\gamma\lambda)^2}\notag\\
& \leq \frac{6R(1+\gamma)\alpha_{k,\tau(\alpha_k)}}{(1-\gamma\lambda)^2}\Eset[\|\theta_{k}\|^2\,|\,\Fcal_{k-\tau(\alpha_{k})}] + \frac{12R(R+1)\alpha_{k,\tau(\alpha_k)}}{(1-\gamma\lambda)^2},
\end{align*}
where in the last inequality we use $2x\leq x^2 + 1$.
\end{proof}

Finally, using Lemmas \ref{lem_tv:xbar_bound}--\ref{lem_tv:opt_bound} we now show Lemma \ref{lem_tv:bias}. 

\begin{proof}[Proof of Lemma \ref{lem_tv:bias}]
First we consider 
\begin{align*}
&\Eset\left[(\btheta_{k}-\theta^*)^{T}(\Abf(X_{k})\btheta_{k} -\Abf\btheta_{k})\,|\,\Fcal_{k-\tau(\alpha_{k})}\right]\\
&\quad = \Eset\left[(\btheta_{k-\tau(\alpha_{k})}-\theta^*)^{T}(\Abf(X_{k})-\Abf)\btheta_{k-\tau(\alpha_{k})}\,|\,\Fcal_{k-\tau(\alpha_{k})}\right]\notag\\ 
&\quad\qquad + \Eset\left[(\btheta_{k-\tau(\alpha_{k})}-\theta^*)^{T}(\Abf(X_{k})-\Abf)(\btheta_{k}-\btheta_{k-\tau(\alpha_{k})})\,|\,\Fcal_{k-\tau(\alpha_{k})}\right]\notag\\ 
&\quad \qquad + \Eset\left[(\btheta_{k}-\btheta_{k-\tau(\alpha_{k})})^{T}(\Abf(X_{k}) -\Abf)\btheta_{k-\tau(\alpha_{k})}\,|\,\Fcal_{k-\tau(\alpha_{k})}\right]\notag\\ 
&\quad \qquad + \Eset\left[(\btheta_{k}-\btheta_{k-\tau(\alpha_{k})})^{T}(\Abf(X_{k}) -\Abf)(\btheta_{k}-\btheta_{k-\tau(\alpha_{k})})\,|\,\Fcal_{k-\tau(\alpha_{k})}\right], 
\end{align*}
which by using Eqs. \eqref{lem_mixing_tv_bound:Ineq1} and \eqref{lem_tv_opt_bound:Ineq1}--\eqref{lem_tv_opt_bound:Ineq3} we have
\begin{align}
&\Big|\Eset\left[(\btheta_{k}-\theta^*)^{T}(\Abf(X_{k})\btheta_{k} -\Abf\btheta_{k})\,|\,\Fcal_{k-\tau(\alpha_{k})}\right]\Big|\notag\\
& \leq 27\alpha_{k}\Eset\left[\|\btheta_{k}\|^2\,|\,\Fcal_{k-\tau(\alpha_{k})}\right] + (24R^2 + \|\theta^*\|^2)\alpha_{k}\notag\\
&\quad\qquad + \frac{108(1+\gamma)^2\alpha_{k,\tau(\alpha_k)}}{(1-\gamma\lambda)^2}\Eset\left[\|\btheta_{k}\|^2\,|\,\Fcal_{k-\tau(\alpha_{k})}\right] +  \frac{100(1+\gamma)^2(R+\|\theta^*\|)^2\alpha_{k,\tau(\alpha_k)}}{(1-\gamma\lambda)^2}\notag\\
&\quad\qquad + \frac{36(R+2)(1+\gamma)^2\alpha_{k,\tau(\alpha_k)}}{(1-\gamma\lambda)^2}\Eset\left[\|\btheta_{k}\|^2\,|\,\Fcal_{k-\tau(\alpha_{k})}\right] + \frac{32R(R+1)^2(1+\gamma)^2\alpha_{k,\tau(\alpha_k)}}{(1-\gamma\lambda)^2}\notag\\
&\qquad\quad + \frac{48(1+\gamma)^2\alpha_{k,\tau(\alpha_k)}}{(1-\gamma\lambda)^2} \Eset\left[\|\btheta_{k}\|^2 \,|\,\Fcal_{k-\tau(\alpha_{k})}\right] + \frac{48R^2\alpha_{k,\tau(\alpha_k)}}{(1-\gamma\lambda)^2}\notag\\
& \leq \left(27 + \frac{(228 + 36R)(1+\gamma)^2}{(1-\gamma\lambda)^2}\right)\tau(\alpha_{k})\alpha_{k-\tau(\alpha_k)} \Eset\left[\|\btheta_{k}\|^2 \,|\,\Fcal_{k-\tau(\alpha_{k})}\right] + (24R^2+\|\theta^*\|^2 )\tau(\alpha_{k})\alpha_{k-\tau(\alpha_k)}\notag\\
&\quad\qquad + \frac{\Big(48R^2+32R(R+1)^2+100(R+\|\theta^*\|)^2\Big)(1+\gamma)^2}{(1-\gamma\lambda)^2}\tau(\alpha_{k})\alpha_{k-\tau(\alpha_k)},\label{lem_tv_bias:Eq1a}
\end{align}
where in the last inequality we use $\alpha_{k}\leq \alpha_{k-\tau(\alpha_k)}$ and $\alpha_{k,\tau(\alpha_k)}\leq \tau(\alpha_k)\alpha_{k-\tau(\alpha_k)}$. Similarly, we consider 
\begin{align*}
(\btheta_{k}-\theta^*)^{T}(\bbar(X_{k}) - b) = (\btheta_{k}-\btheta_{k-\tau(\alpha_{k})})^{T}(\bbar(X_{k}) - b) + (\btheta_{k-\tau(\alpha_{k})}-\theta^*)^{T}(\bbar(X_{k}) - b),    
\end{align*}
which by using Eqs.\ \eqref{lem_mixing_tv_bound:Ineq2} and \eqref{lem_tv_opt_bound:Ineq4} yields
\begin{align}
&\Big|\Eset\left[(\btheta_{k}-\theta^*)^{T}(\bbar(X_{k}) - b)\,|\,\Fcal_{k-\tau(\alpha_{k})}\right]\Big|\notag\\
&\quad \leq 9\alpha_{k}\Eset[\|\btheta_{k}\|^2\,|\,\Fcal_{k-\tau(\alpha_{k})}] + \left(8R^2 + 1 + \|\theta^*\|\right)\alpha_{k}\notag\\ 
&\quad\qquad + \frac{6R(1+\gamma)}{(1-\gamma\lambda)^2}\alpha_{k,\tau(\alpha_{k})}\Eset[\|\theta_{k}\|^2\,|\,\Fcal_{k-\tau(\alpha_{k})}] + \frac{12R(R+1)}{(1-\gamma\lambda)^2}\alpha_{k,\tau(\alpha_{k})}\notag\\
&\quad = \left(9 + \frac{6R(1+\gamma)}{(1-\gamma\lambda)^2}\right)\tau(\alpha_k)\alpha_{k-\tau(\alpha_{k})}\Eset[\|\theta_{k}\|^2\,|\,\Fcal_{k-\tau(\alpha_{k})}]\notag\\
&\quad \qquad + \left(\left(8R^2 + 1 + \|\theta^*\|\right) + \frac{12R(R+1)}{(1-\gamma\lambda)^2} \right)\tau(\alpha_k)\alpha_{k-\tau(\alpha_{k})}.\label{lem_tv_bias:Eq1b}
\end{align}
Using Eqs.\ \eqref{lem_tv_bias:Eq1a} and \eqref{lem_tv_bias:Eq1b} we obtain Eq.\ \eqref{lem_tv_bias:Ineq}, i.e.,
\begin{align*}
&\left|\Eset\left[(\btheta_{k}-\theta^*)^{T}(\Abf(X_{k})\btheta_{k} -\Abf\btheta_{k} + \bbar(X_{k}) - b)\,|\,\Fcal_{k-\tau(\alpha_{k})}\right]\right|\notag\\
&\quad \leq \left(27 + \frac{(228 + 36R)(1+\gamma)^2}{(1-\gamma\lambda)^2}\right)\tau(\alpha_k)\alpha_{k-\tau(\alpha_{k})} \Eset\left[\|\btheta_{k}\|^2 \,|\,\Fcal_{k-\tau(\alpha_{k})}\right]\notag\\
&\quad\qquad + \left(24R^2+\|\theta^*\|^2\right)\tau(\alpha_k)\alpha_{k-\tau(\alpha_{k})}\notag\\
&\quad\qquad + \left( \frac{\Big(48R^2+32R(R+1)^2+100(R+\|\theta^*\|)^2\Big)(1+\gamma)^2}{(1-\gamma\lambda)^2}\right)\tau(\alpha_k)\alpha_{k-\tau(\alpha_{k})}\notag\\
&\quad\qquad + \left(9 + \frac{6R(1+\gamma)}{(1-\gamma\lambda)^2}\right)\tau(\alpha_k)\alpha_{k-\tau(\alpha_{k})}\Eset[\|\theta_{k}\|^2\,|\,\Fcal_{k-\tau(\alpha_{k})}]\notag\\ 
&\quad\qquad + \left(\left(8R^2 + 1 + \|\theta^*\|\right) + \frac{12R(R+1)}{(1-\gamma\lambda)^2} \right)\tau(\alpha_k)\alpha_{k-\tau(\alpha_{k})}\notag\\
&\quad \leq \left(36 + \frac{(228 + 42R)(1+\gamma)^2}{(1-\gamma\lambda)^2}\right)\tau(\alpha_k)\alpha_{k-\tau(\alpha_{k})} \Eset\left[\|\btheta_{k}\|^2 \,|\,\Fcal_{k-\tau(\alpha_{k})}\right]\notag\\
&\quad\qquad + \left(32R^2+ 2\|\theta^*\|^2 + 1\right)\tau(\alpha_k)\alpha_{k-\tau(\alpha_{k})}\notag\\
&\quad\qquad + \left( \frac{\Big(48R^2+32(R+1)^3+100(R+\|\theta^*\|)^2\Big)(1+\gamma)^2}{(1-\gamma\lambda)^2}\right)\tau(\alpha_k)\alpha_{k-\tau(\alpha_{k})}.
\end{align*}
\end{proof}

\end{document}